\documentclass[12pt]{article}

\usepackage{graphicx}
\usepackage[utf8]{inputenc}
\usepackage[numbers]{natbib}
\usepackage{amsmath}
\usepackage{bbm}
\usepackage{amssymb,amsthm}
\usepackage{mathrsfs}
\usepackage{indentfirst}
\usepackage{geometry}
\usepackage{url}
\usepackage{comment}
\usepackage{transparent}
\usepackage{enumitem}
\usepackage{xcolor}

\allowdisplaybreaks[4]
\numberwithin{equation}{section}
\newtheorem{theorem}{\textbf{Theorem}}[section]
\newtheorem{assumption}[theorem]{Assumption}
\newtheorem{definition}[theorem]{Definition}
\newtheorem{lemma}[theorem]{\textbf{Lemma}}
\newtheorem{proposition}[theorem]{\textbf{Proposition}}
\newtheorem{corollary}[theorem]{\textbf{Corollary}}
\newtheorem{remark}[theorem]{\textbf{Remark}}
\usepackage[colorlinks=true, allcolors=blue]{hyperref}

\def\r{{\mathbb R}}
\def\h{{\mathbb H}}
\def\e{{\mathbb E}}
\def\p{{\mathbb P}}
\def\z{{\mathbb Z}}
\def\N{{\mathbb N}}
\def\dd{\mathrm{d}}
\def\cD{\mathcal D}
\def\cW{\mathcal W}
\def\cS{\mathcal S}
\def\tcS{\widetilde{\mathcal S}}
\def\cY{\mathcal Y}
\def\bD{\mathbb D}
\def\cF{\mathcal F}
\def\cG{\mathcal G}
\def\cI{\mathcal I}
\def\hsig{\hat \sigma}
\def\cShat{\widehat {\mathcal S}}

\setlist[enumerate]{
    topsep=6pt,
    itemsep=4pt,
    partopsep=0pt,
    parsep=4pt
}

\title{\Large{\textbf{Continuous Stochastic Flows Driven by White Noise and Their Duals}}}
\author{Yaolin Yu\footnote{\scriptsize SMS, Fudan University, China, \texttt{ylyu23@m.fudan.edu.cn}}}
\date{}

\begin{document}

\maketitle

\begin{abstract}
We study a class of continuous stochastic flows driven by a space-time white noise and characterize their dual flows by explicit stochastic differential equations. A key ingredient of the proof is the convergence of solutions under coefficient approximations. As an application, we derive the dual flows in two illustrative examples, the squared Bessel flow and the Jacobi flow. We also introduce a new polynomially self-repelling (PSR) flow and establish its self-duality.
\end{abstract}

{\bf Classification}. 60J65, 60J55, 60H10.

{\bf Keywords}. stochastic flow, stochastic differential equation, white noise, duality.

\section{Introduction}
Let $\cW$ be a standard space-time white noise on $\r^2$ (see, e.g., \cite{Walsh1986AnIT}). Let $b:\r^2 \to\r$ and $\sigma:\r^3 \to\r$ be Borel functions. We study a class of stochastic flows driven by $\cW$, where each flow line starting from $(a,r)\in \r^2$ is the trajectory $[r,\infty)\ni x\mapsto \cS_{r,x}(a)$ given by the solution to the stochastic differential equation (SDE)
\begin{align}\label{eq:SDE_gener}
	\cS_{r,x}(a) = a +\int_r^x \!\! \int_{\r} \sigma(\cS_{r,u}(a), u, s) \cW(\dd s, \dd u)+ \int_r^x b(\cS_{r,u}(a),u)  \dd u.
\end{align}
Define
\begin{equation*}
	\bar{\sigma}(x,t) := \int_{\r} \sigma(x,t,s)^2 \dd s.
\end{equation*}
Throughout this paper, we work under the following hypothesis on the coefficients.

\begin{assumption}\label{assu:H}
	The mapping $(x,t) \mapsto b(x,t) $ is continuous on $\r^2$, and $\bar{\sigma}(x,t) < \infty $ for all $ (x,t)\in\r^2 $. Furthermore, for any compact sets $ K_1, K_2 \subset \r $,
	\begin{align}\label{assu:conti_t}
		\sup_{x\in K_1}\sup_{t\in K_2}\int_\r \bigl(\sigma(x,t+h,s)-\sigma(x,t,s)\bigr)^2 \dd s \to 0
		\quad \text{as } h \to 0 .
	\end{align}
\end{assumption}

Before giving a rigorous definition of stochastic flows, we impose further structural assumptions on the coefficients $(\sigma,b)$. These assumptions are standard for SDEs~\eqref{eq:SDE_gener} and the forthcoming~\eqref{eq:back_flow} with possibly non-Lipschitz coefficients (see, e.g., \cite{SDE81,kara98}). 
For any interval $[T_1,T_2] \subset \r$, the constants $K_b$, $K_\sigma$ and the functions $\rho_m$, $r_m$, and $r'_m$ appearing in the assumptions below are taken to depend on $T_1$ and $T_2$.

\smallskip
\begin{enumerate}[label=\textnormal{(1.\alph*)}, nosep]
    \item There is a constant $K_b>0$ such that
    \begin{align*}
        |b(x,t)|\le K_b(1+|x|), \quad \forall x\in \r, t\in [T_1,T_2].
    \end{align*} \label{item:1a}
    
    \item For each $m \geq 1$ there is a nonnegative nondecreasing function $\rho_m$ on $\r_{+}$ such that $\int_{0_+} \rho_m(z)^{-2} \dd z=\infty$, and for all $x, y \in [-m,m]$ and $t\in [T_1,T_2]$,
    $$
    \int_{\r}|\sigma(x, t, s)-\sigma(y, t, s)|^2 \dd s \leq \rho_m(|x-y|)^2.
    $$\label{item:1b}
    
    \item For each $t\in \r$, $\bar{\sigma}(\cdot,t)\in C^1(\r)$ and there is a constant $K_\sigma>0$ such that
    \begin{align*}
        \left|\partial_x \bar{\sigma}(x,t)\right|\le K_{\sigma}(1+|x|),\quad  \forall x\in \r, t\in [T_1,T_2].
    \end{align*}\label{item:1c}
    
    \item For each $m\ge1$, there are nondecreasing concave functions $r_m$ and $r'_m$ on $\r_+$ such that $\int_{0_+}r_m(z)^{-1}\dd z=\int_{0_+}r'_m(z)^{-1}\dd z=\infty$, and, for all $x,y\in[-m,m]$ and $t\in[T_1,T_2]$,
    $$
    |b(x,t)-b(y,t)|\le r_m(|x-y|),
    $$
    $$
    \left|-b(x,t)+b(y,t)
    +\tfrac12\bigl(
        \partial_x\bar\sigma(x,t)
        -\partial_y\bar\sigma(y,t)
    \bigr)\right|
    \le r'_m(|x-y|).
    $$\label{item:1d}
\end{enumerate}

Under Assumption~\ref{assu:H} and the above conditions, Lemma~\ref{app:lem_conti} implies that the maps $(x,t)\mapsto \bar\sigma(x,t)$ and $(x,t)\mapsto -b(x,t)+\frac{1}{2}\partial_x\bar\sigma(x,t)$ are continuous. Moreover, conditions~\ref{item:1a} and~\ref{item:1c} ensure weak existence for both~\eqref{eq:SDE_gener} and~\eqref{eq:back_flow}; see \cite[Chapter~3, Section~3]{Skorokhod65}. Pathwise uniqueness is proved in Section~\ref{app:path_unique}. The Yamada--Watanabe theorem (see, e.g., \cite{YW71,strong75}) then yields unique strong solutions to both SDEs.

\begin{definition}\label{def:flow}
A \emph{stochastic flow} driven by $\cW$ is a family of processes 
$$
\cS = \left\{\bigl(\cS_{r,x}(a),\, x \ge r\bigr)\right\}_{(a,r) \in \r^2}
$$
such that, for each $(a,r) \in \r^2$, the process $(\cS_{r,x}(a),\, x \ge r)$ is the almost sure unique strong solution to~\eqref{eq:SDE_gener}, and the following properties hold almost surely:
\begin{enumerate}[label=\textnormal{(R\arabic*)}]
    \item $\cS_{r,r}(a) = a$ for all $(a,r) \in \r^2$. \label{item:R1}
    \item For every $r \le x$, $a \mapsto \cS_{r,x}(a)$ is c\`adl\`ag. \label{item:R2}
    \item For any $(a_1,r_1), (a_2,r_2) \in \r^2$ and any $z \ge y \ge \max\{r_1,r_2\}$,
    $$
        \cS_{r_1,y}(a_1) < \cS_{r_2,y}(a_2) 
        \ \Longrightarrow\ 
        \cS_{r_1,z}(a_1) \le \cS_{r_2,z}(a_2).
    $$ \label{item:R3}
\end{enumerate}

A stochastic flow is said to be \emph{continuous} if, almost surely, all flow lines are continuous. 
It is said to be \emph{unique} if any other stochastic flow associated with~\eqref{eq:SDE_gener} satisfying the above properties is equal to $\cS$ almost surely.
\end{definition}

The following theorem states the existence, uniqueness, and continuity of the associated stochastic flow. 
Existence and uniqueness follow from the stronger result in Theorem~\ref{thm:ex}, while continuity will be proved later in Proposition~\ref{prop:conti}.

\begin{theorem}\label{thm:ex_intro}
Suppose that $(\sigma,b)$ satisfy Assumption~\ref{assu:H} and conditions~\ref{item:1a}--\ref{item:1d}. Then the associated stochastic flow $\cS$ exists, is unique, and is continuous. 
\end{theorem}

\bigskip

Let $\cS$ be the stochastic flow introduced in Definition~\ref{def:flow}, and suppose that it exists. For each $(b,r) \in \r^2$ and $x\le r$, define the backward line by
\begin{equation}\label{eq:def_dual}
    \cS^*_{r,x}(b) := \inf \{ a \in \r :\, \cS_{x,r}(a) > b \}.
\end{equation}
The collection $\cS^* = \{(\cS^*_{r,x}(b),\, x \le r)\}_{(b,r) \in \r^2}$ is called the \emph{dual flow} (or \emph{backward flow}) associated with $\cS$. 
A natural question is whether the family $\{(\cS^*_{-r,-x}(a),\, x \ge r)\}_{(a,r) \in \r^2}$ also constitutes a stochastic flow. The following theorem provides an affirmative answer, which is the main result of the paper.

\begin{theorem}\label{thm:backward_flow_SDE}
Suppose that $(\sigma,b)$ satisfy Assumption~\ref{assu:H} and conditions~\ref{item:1a}--\ref{item:1d}, and let $\cS$ be the associated stochastic flow. Then, for any $(a,r) \in \r^2$, $\cS^*_{-r,-\cdot}(a)$ is the almost sure unique strong solution to
\begin{align}\label{eq:back_flow}
    \nonumber \cS^*_{-r,-x}(a) 
    = a  &- \int_r^x\!\! \int_{\r} \sigma\big( \cS^*_{-r,-u}(a), -u, s \big) \cW^*(\dd s, \dd u) - \int_r^x b\big( \cS^*_{-r,-u}(a), -u \big) \dd u \\
    & + \frac{1}{2} \int_r^x  \partial_x \bar\sigma \big( \cS^*_{-r,-u}(a), -u \big) \dd u, \qquad x \ge r,
\end{align}
where $\cW^*$ is the image of $\cW$ under the map $(a, r) \mapsto(a,-r)$. Moreover, the collection
$\{(\cS^*_{-r,-x}(a),\, x \ge r)\}_{(a,r) \in \r^2}$ forms a stochastic flow. In particular, $(\cS^*)^*=\cS$.
\end{theorem}

As will be reviewed in Section~\ref{sec:homeo}, Kunita~\cite{Ku90} shows that for a flow of homeomorphisms—arising under suitable smoothness assumptions on the coefficients $(\sigma,b)$—the associated backward flow satisfies the SDE~\eqref{eq:back_flow}. In this setting, Theorem~\ref{thm:backward_flow_SDE} follows directly from Kunita's theory; our result extends this to coalescing flows.
Our approach, based on approximating coalescing flows by flows of homeomorphisms, is inspired by techniques appearing in the work of Burdzy and Chen~\cite{Burdzy_Chen} and A\"{\i}d\'{e}kon et al.~\cite{2024meetingsquaredbesselflow} in the context of skew Brownian flows.

While our construction is formulated on $\r^2$, future applications require a more flexible framework. For instance, the local time flows studied in~\cite[Section~2]{elie23} are naturally supported on the half-plane $\h:=\r_+\times\r$. This motivates us to consider stochastic flows on strip-like domains with general boundary behavior. In Section~\ref{sec:strip}, we develop this extension and establish the corresponding duality relations.

\bigskip

To illustrate our general result, we examine two well-known examples: the squared Bessel (BESQ) flow and the Jacobi flow (see Theorem~\ref{thm:back_BESQ}). The dual of a general BESQ$^\delta$ flow was identified in \cite[Proposition~3.8]{elie23} via an embedding into the perturbed reflecting Brownian motion (PRBM). We present an alternative approach that recovers the same dual flow. 
The second example concerns the Jacobi flow on $(0,1)$, as studied in A\"{\i}d\'{e}kon et al.~\cite{elie23}. To the best of our knowledge, its dual flow has not previously been explicitly identified, and we give its precise characterization in this paper.

Beyond these examples, we introduce, for each $\alpha>0$, a new model of stochastic flow called the \emph{polynomially self-repelling flow with index $\alpha$ (PSR$^\alpha$ flow)}. The terminology is motivated by its connection to the scaling limit of the polynomially self-repelling walk studied in~\cite{Elena23, Toth96}. Although this connection will be explored in forthcoming work, our focus here is to establish the existence and uniqueness of this flow and derive an explicit SDE for its backward flow, leading to a self-duality property.

\paragraph{Notation} 
Let $\bD$ denote the set of dyadic numbers, i.e., numbers of the form $k2^{-n}$ for $k \in \mathbb{Z}$ and $n \in \mathbb{N}$, and let $\bD_+$ denote the set of positive dyadic numbers. Set $\r_+ := [0,\infty)$ and write $\h := \r_+ \times \r$ for the half-plane. We use $a \vee b := \max\{a,b\}$ and $a \wedge b := \min\{a,b\}$.

\paragraph{Related works} 
Le Jan and Raimond~\cite{Le_Jan} studied coalescing flows and flows of probability kernels, going beyond the flows of diffeomorphisms. Our SDE~\eqref{eq:SDE_gener} can be seen as a time-inhomogeneous, jump-free variant of Dawson and Li~\cite[Equation (2.1)]{Dawson12}.
Dual objects associated with flows arise in multiple contexts. Fontes et al.~\cite{Fontes2003TheBW} showed that the dual of the Brownian web is itself a Brownian web. Bertoin and Le Gall~\cite{BL03,BL05} showed that generalized Fleming–Viot processes are dual to flows of bridges associated with exchangeable coalescents. In particular, \cite{BL05} proved a distributional Jacobi$(2,2)$/Jacobi$(0,0)$ duality between the inverse and forward motions. The formulation~\eqref{eq:def_dual} also appeared in the construction of the inverse of continuous-state branching processes, as discussed by Foucart et al.~\cite{Foucart2018CoalescencesIC}, where the notion of Siegmund duality was first developed by Siegmund~\cite{Siegmund1976TheEO}; see also Clifford and Sudbury~\cite{CS85}. More recently, duality has played an important role in stochastic flows driven by white noise. A\"{\i}d\'{e}kon et al.~\cite{elie23,2024meetingsquaredbesselflow} studied the dual of the local time flow and used it to study bifurcation events, namely points where multiple flow lines emerge.

\paragraph{Structure of the paper}
In Section~\ref{sec:homeo}, we introduce stochastic flows of homeomorphisms together with their associated backward flows. Section~\ref{sec:appro} establishes convergence results for coefficient and initial-condition approximations. The proofs of the main results, Theorems~\ref{thm:ex_intro} and~\ref{thm:backward_flow_SDE}, except for continuity, are given in Section~\ref{sec:exist}. Section~\ref{sec:strip} studies stochastic flows on strips with various boundary conditions, including their continuity properties. In Section~\ref{sec:apply}, we apply the main results to two examples of dual flows and introduce the PSR flow. Finally, Appendix~\ref{app:conv} contains auxiliary lemmas for the convergence results in Section~\ref{subsec:conv_re}, while Appendix~\ref{app:appro} details the construction of the coefficient approximations.

\paragraph{Acknowledgements}
All original mathematical ideas, results, and proofs presented in this paper were developed by the author without the use of generative AI. We are grateful to Dongjian Qian and Chengshi Wang for helpful discussions, and to Elie A\"{\i}d\'{e}kon for carefully reading the entire manuscript.

\section{Stochastic flow of homeomorphisms}\label{sec:homeo}
For each fixed $(a,r) \in \r^2$, the strong solution $\bigl(\cS_{r,x}(a),\, x \ge r\bigr)$ to~\eqref{eq:SDE_gener} is a semimartingale with respect to the filtration generated by $\cW$. 
Define a continuous semimartingale
\begin{align}\label{eq:F}
    F(x,t):=\int_r^t\!\! \int_{\r} \sigma(x,u,s)\cW(\dd s, \dd u)+\int_r^t b(x,u) \dd u,
\end{align}
so that the SDE~\eqref{eq:SDE_gener} can be equivalently written as
\begin{align}\label{eq:SDE_gener_var}
    \cS_{r,x}(a)=a+\int_r^x F(\cS_{r,u}(a), \dd u), \quad x\ge r.
\end{align}
\noindent
 Let $(a(x,y,t), b(x,t), A_t)$ denote the local characteristics of $F(x,t)$ in the sense of \cite[Section 3.2]{Ku90}. They are given explicitly by
\begin{equation}\label{eq:local_char}
    \bigl(a(x,y,t),\, b(x,t),\, A_t\bigr)
    = \left( \int_{\r} \sigma(x,t,s)\sigma(y,t,s)\dd s,\, b(x,t),\, t \right).
\end{equation}

In this section, we recall the definition of a Brownian flow from \cite[Section~4]{Ku90} and, under suitable regularity assumptions on the coefficients $\sigma$ and $b$ in~\eqref{eq:SDE_gener}, state the SDE satisfied by the backward flow of homeomorphisms.

\subsection{Regularity conditions on local characteristics} 

We impose the following regularity conditions on $\sigma$ and $b$. Let $a$ be defined in~\eqref{eq:local_char}.

\begin{assumption}\label{assu:A}
Here, $K$ denotes a finite positive constant that may vary from line to line, depending on the interval $[T_1,T_2]\subset \r$.

\smallskip
\begin{enumerate}[label=(\textbf{A}\arabic*)]
    \item The function $b$ satisfies condition~\ref{item:1a}, $b(\cdot, t) \in C^1(\r)$ for each $t\in \r$, and for all $x,y\in \r$ and $t \in [T_1,T_2]$,
    \begin{gather*}
      |\partial_x b(x,t)|\le K,\qquad |\partial_x b(x,t) - \partial_y b(y,t)| \le K |x - y|.
    \end{gather*}\label{item:A1}

    \item The function $\sigma: \r^3 \to \r$ is smooth in the first variable, and for all $x,y\in \r$, $t \in [T_1,T_2]$ and $\alpha\in \{1,2,3\}$,
    \begin{gather*}
    \frac{|a(x,y,t)|}{(1 + |x|)(1 + |y|)} \le K, \qquad
    \int_\r \bigl(\partial_x^\alpha\sigma(x,t,s)\bigr)^2\dd s \le K.
    \end{gather*}\label{item:A2}
\end{enumerate}    
\end{assumption}

\begin{lemma}\label{lm:es_conA}
Under Assumptions~\ref{assu:H} and~\ref{assu:A}, the following properties hold.
\begin{enumerate}[label=(\roman*)]
    \item For any $x,y,t\in \r$ and $\alpha,\beta\in \{1,2\}$,
    $$
    \partial_x^\alpha \partial_y^\beta a(x,y,t) = \int_\r \partial_x^\alpha \sigma(x,t,s) \partial_y^\beta \sigma(y,t,s)\dd s.
    $$
    In particular, $\partial_x \bar{\sigma}(x,t) = 2 \int_\r \sigma(x,t,s)\partial_x\sigma(x,t,s)\dd s $.
        
    \item The functions $a$ and $b$ belong to the classes $\widetilde{C}^{2,1}_{ub}$ and $C^{1,1}_{ub}$, respectively; see \cite[pp.~334--335]{Ku90} for the definitions of these spaces.
        
    \item The coefficients $(\sigma,b)$ satisfy conditions~\ref{item:1a}--\ref{item:1d}.
\end{enumerate}
\end{lemma}

\begin{proof}
    (i) We treat only the case $\alpha=1$ and $\beta=0$, as higher-order derivatives follow analogously. For $h>0$, 
    \begin{align*}
        a(x+h,y,t)-a(x,y,t)=&\, \int_\r \bigl(\sigma(x+h,t,s)-\sigma(x,t,s)\bigr)\sigma(y,t,s) \dd s\\
        =&\, \int_\r \int_x^{x+h} \partial_\xi \sigma(\xi,t,s) \dd \xi \, \sigma(y,t,s) \dd s.
    \end{align*}
    By the Cauchy--Schwarz inequality and~\ref{item:A2} (which entails $\bar{\sigma}(y,t) \le K(1+|y|)^2$),
    \begin{align*}
        \int_\r \int_x^{x+h} |\partial_\xi \sigma(\xi,t,s) \sigma(y,t,s)| \dd \xi \dd s 
        &\le \int_x^{x+h} \Bigl( \int_\r |\partial_\xi \sigma|^2 \dd s \cdot \bar{\sigma}(y,t) \Bigr)^{1/2} \dd \xi \le K h (1+|y|).
    \end{align*}
    This justifies the use of Fubini's theorem. Hence,
    \begin{align*}
        a(x+h,y,t)-a(x,y,t)=\int_x^{x+h}\!\!\int_\r \partial_\xi \sigma(\xi,t,s) \sigma(y,t,s) \dd s\dd \xi.
    \end{align*}
    To conclude via the fundamental theorem of calculus, we only need the continuity of $\xi\mapsto \int_\r \partial_\xi \sigma(\xi,t,s) \sigma(y,t,s) \dd s$. For any $\xi_1<\xi_2$, applying Cauchy--Schwarz again gives
    \begin{align*}
        \left|\int_\r \bigl(\partial_\xi\sigma(\xi_1,t,s)- \partial_\xi\sigma(\xi_2,t,s)\bigr) \sigma(y,t,s) \dd s\right|^2\le&\; 
        \bar\sigma(y,t) \int_\r \Bigl(\int_{\xi_1}^{\xi_2} \left|\partial_\xi^2 \sigma(\xi,t,s)\right|\dd \xi\Bigr)^2\dd s \\
        \le &\; \bar\sigma(y,t) (\xi_2-\xi_1) \int_{\xi_1}^{\xi_2}\! \int_\r  \left|\partial_\xi^2 \sigma(\xi,t,s)\right|^2 \dd s \dd \xi\\
        \le &\; K^2(1+|y|)^2 (\xi_2-\xi_1)^2,
    \end{align*}
    where the last bound uses~\ref{item:A2}. The continuity follows immediately.

    \smallskip 
    (ii) By (i), for each fixed $t$, the derivatives $\partial_x^\alpha \partial_y^\beta a(\cdot,\cdot,t)$ exist and are continuous for $\alpha,\beta\le 2$, hence $a(\cdot,\cdot,t)\in C^2(\r^2)$. Applying Cauchy--Schwarz together with $\int_\r (\partial_x^\alpha \sigma)^2 \dd s \le K$ for $\alpha\le 2$ from~\ref{item:A2}, we deduce that $|\partial_x \partial_y a(x,y,t)|$ and $|\partial_x^2 \partial_y^2 a(x,y,t)|$ are uniformly bounded by $K$ for all $x,y\in\r$ and $t\in[T_1,T_2]$. 
    Using the bound for $\alpha=3$ in~\ref{item:A2} and Minkowski's inequality, for any $x\ne x'$, $y\ne y'$ and $t\in[T_1,T_2]$,
    \begin{align*}
    &\; |\partial_x^2 \partial_y^2 a(x,y,t) - \partial_x^2 \partial_y^2 a(x',y,t) - \partial_x^2 \partial_y^2 a(x,y',t) + \partial_x^2 \partial_y^2 a(x',y',t)| \\
    = &\; \left| \int_\r \left( \int_{x'}^x \partial_u^3 \sigma(u,t,s) \dd u \right) \left( \int_{y'}^y \partial_v^3 \sigma(v,t,s) \dd v \right) \dd s \right| \\
    \le &\; \left( \int_\r \left( \int_{x'}^x \partial_u^3 \sigma(u,t,s) \dd u \right)^2 \dd s \right)^{\!1/2} \left( \int_\r \left( \int_{y'}^y \partial_v^3 \sigma(v,t,s) \dd v \right)^2 \dd s \right)^{\!1/2} \\
    \le &\; \left| \int_{x'}^x \left( \int_\r \bigl(\partial_u^3 \sigma(u,t,s)\bigr)^2 \dd s \right)^{\!1/2} \dd u \right| \left| \int_{y'}^y \left( \int_\r \bigl(\partial_v^3 \sigma(v,t,s)\bigr)^2 \dd s \right)^{\!1/2} \dd v \right| \\
    \le &\; K |x - x'| |y - y'|.
    \end{align*}
    Moreover, by the uniform bound $|\partial_x \partial_y a|\le K$,
    \begin{align}\label{ineq:bo}
        \nonumber \int_{\r} |\sigma(x,t,s) - \sigma(y,t,s)|^2 \dd s 
        &= a(x,x,t) - 2a(x,y,t) + a(y,y,t) \\
        &= \int_y^x \int_y^x \partial_u \partial_v a(u,v,t) \dd u \dd v\le K |x-y|^2.
    \end{align}

    We next prove the continuity of $a$. For $(x,y,t),(x',y',t')\in \r^2\times[T_1,T_2]$,
    \begin{align}\label{eq:diff_a}
        |a(x,y,t)-a(x',y',t')|\le |a(x,y,t)-a(x,y,t')|+|a(x,y,t')-a(x',y',t')|.
    \end{align}
    Let $\delta(z;t,t') := \int_\r (\sigma(z,t,s) - \sigma(z,t',s))^2 \dd s$. By Cauchy--Schwarz, 
    \begin{align*}
        |a(x,y,t) - a(x,y,t')| \le \sqrt{K}(1+|x|)\sqrt{\delta(y; t, t')} + \sqrt{K}(1+|y|)\sqrt{\delta(x; t, t')},
    \end{align*}
    which vanishes as $t\to t'$ by~\eqref{assu:conti_t}. For the second term on the right-hand side of~\eqref{eq:diff_a}, applying~\eqref{ineq:bo} and the Cauchy--Schwarz inequality yields
    \begin{align*}
        |a(x,y,t') - a(x',y',t')| &\le \int_\r |\sigma(x,t',s)| |\sigma(y,t',s) - \sigma(y',t',s)| \dd s \\
        &\quad + \int_\r |\sigma(y',t',s)| |\sigma(x,t',s) - \sigma(x',t',s)| \dd s \\
        &\le \sqrt{K}(1+|x|) \cdot \sqrt{K} |y - y'| + \sqrt{K}(1+|y'|) \cdot \sqrt{K} |x - x'|.
    \end{align*}
    As $(x,y,t) \to (x',y',t')$, each component converges to $0$, ensuring the joint continuity of $a$. Hence $a \in \widetilde{C}^{2,1}_{ub}$. Finally, $b \in C^{1,1}_{ub}$ follows directly from~\ref{item:A1} and the continuity of $b$.

    \smallskip
    (iii) Assumption~\ref{item:A1} implies \ref{item:1a} and the bound for $b$ in~\ref{item:1d}. By~\eqref{ineq:bo}, condition~\ref{item:1b} holds with $\rho_m(z) := \sqrt{K} z$. Applying the Cauchy--Schwarz inequality and~\ref{item:A2} yields $|\partial_x \bar{\sigma}| = 2 \left| \int_\r \sigma \partial_x\sigma\dd s \right| \le 2K(1+|x|)$, which verifies condition~\ref{item:1c}. Finally, $a(\cdot,\cdot,t) \in C^2(\r^2)$ ensures $\bar{\sigma}(\cdot,t) \in C^1(\r)$, so that $\partial_x \bar{\sigma}$ is locally Lipschitz and \ref{item:1d} follows.
\end{proof}

\subsection{Brownian flows}

Recall from \cite[Section 4.1]{Ku90} that a Brownian flow $\cS_{r,x}$ is a stochastic flow of homeomorphisms satisfying the following properties almost surely:
\begin{enumerate}[label=(\roman*)]
    \item $\cS_{r,y}=\cS_{x,y}\circ \cS_{r,x}$ holds for all $r,x,y$.
    \item $\cS_{r,r}$ is the identity map for all $r$.
    \item $\cS_{r,x}:\r \to \r$ is an onto homeomorphism for all $r,x$.
    \item For any $0\le t_0<t_1<\cdots<t_n$, $\cS_{t_i,t_{i+1}}$, $i=0,\ldots, n-1$, are independent.
\end{enumerate}

\noindent
For such a Brownian flow $\cS_{r,x}$, the infinitesimal mean and infinitesimal covariance (when they exist) are given respectively by
\begin{gather*}
    \lim_{h\rightarrow 0+} \frac{1}{h} \bigl(\e[\cS_{r,r+h}(x)]-x \bigr),\qquad
    \lim_{h\rightarrow 0+} \frac{1}{h}\, \e\! \left[(\cS_{r,r+h}(x)-x)(\cS_{r,r+h}(y)-y) \right].
\end{gather*}

\begin{proposition}\label{thm:Brownian_flow_for}
    Under Assumptions~\ref{assu:H} and~\ref{assu:A}, the flow governed by~\eqref{eq:SDE_gener_var} is a Brownian flow with infinitesimal mean $b$ and covariance $a$, as specified in~\eqref{eq:local_char}.
\end{proposition}

\begin{proof}
    By Lemma~\ref{lm:es_conA}(ii), to apply \cite[Theorem~4.5.1(ii)]{Ku90}, it remains to verify that $t\mapsto F(x,t)$ has independent increments and takes values in $C^{0,1}$. The former follows from~\eqref{eq:F}. We now verify the latter.\footnote{We refer to \cite[pp.~334]{Ku90} for the definition of $C^{0,1}$. Namely, a function $f$ belongs to $C^{0,1}$ if for every compact set $K\subset \r$, $$\sup_{x\in K}\frac{|f(x)|}{1+|x|}+\sup_{x,y\in K, x\neq y}\frac{|f(x)-f(y)|}{|x-y|}<\infty.$$} Without loss of generality, let $r=0$. Fix $T>0$ and $t\in[0,T]$. By~\ref{item:A1}, the finite variation component is easily controlled. For the Gaussian component, we first show that, for every compact set $K\subset\r$,
    \begin{align}\label{expec_claim}
        \e\left[\sup_{x\in K} |G(x)|\right]<\infty, \qquad \e\left[\sup_{x,y\in K,\, x\neq y} |H(x,y)|\right]<\infty,
    \end{align}
    where
    \begin{gather*}
        G(x):=\int_0^t\!\int_\r\sigma(x,u,s)\cW(\dd s,\dd u),\qquad
        H(x,y):=\int_0^t\!\int_\r q(x,y;u,s)\cW(\dd s,\dd u),\\
        q(x,y;u,s):=\int_0^1\partial_x\sigma(\theta x+(1-\theta)y,u,s)\dd\theta.
    \end{gather*}
    Thus, for $x\ne y$, $H(x,y)=(G(x)-G(y))/(x-y)$ almost surely.
    
    Fix a compact set $K$ and choose $N\ge 1$ such that $K\subset [-N, N]$. Then for any $x,y\in [-N,N]$, we have $|x-y|\le 2N$. By~\eqref{ineq:bo}, there exists a constant $C$ such that
    \begin{align*}
        d_G (x,y):= \e\bigl[(G(x)-G(y))^2\bigr]=\int_0^t \!\! \int_\r \left|\sigma(x,u,s)-\sigma(y,u,s)\right|^2 \dd s \dd u\le 2CNt|x-y|.
    \end{align*}
    Moreover, the bound on $\partial_x^2\sigma$ in~\ref{item:A2} gives
    \begin{align*}
        \|q(x,y;u,\cdot)-q(z,w;u,\cdot)\|_{L^2}
        &\le\sqrt C\int_0^1|\theta(x-z)+(1-\theta)(y-w)|\dd\theta\\
        &\le\frac{\sqrt C}{2}(|x-z|+|y-w|).
    \end{align*}
    Hence, for $(x,y),(z,w)\in [-N,N]^2$,
    \begin{align*}
        d_H((x,y),(z,w))
        &:=\e[(H(x,y)-H(z,w))^2]\\
        &\le\frac{Ct}{2}(|x-z|^2+|y-w|^2)
        \le CNt(|x-z|+|y-w|).
    \end{align*}

    \noindent
    Let $Z_u:=\sqrt{2CNt}B_u$, where $B$ is a standard two-sided Brownian motion. Then $\e[(Z_x-Z_y)^2]=2CNt|x-y|\ge d_G(x,y)$. Applying the Sudakov--Fernique inequality (see, e.g., \cite[Theorem~7.2.11]{High_dim}) to the centered Gaussian processes $G(\cdot)-G(0)$ and $G(0)-G(\cdot)$ yields
    \begin{align*}
        \e\left[\sup_{x\in K}|G(x)|\right]
        &\le\e|G(0)|
        +\e\left[\sup_{x\in [-N,N]}(G(x)-G(0))\right]
        +\e\left[\sup_{x\in [-N,N]}(G(0)-G(x))\right]\\
        &\le\e|G(0)|+2\sqrt{2CNt}\,\e\left[\sup_{|u|\le N}B_u\right]<\infty.
    \end{align*}
    Here the two suprema are nonnegative because $0\in [-N,N]$, and $\e|G(0)|<\infty$ by~\ref{item:A2}. Similarly, let $W_{u,v}:=2\sqrt{CNt}(B^1_u+B^2_v)$, where $B^1$ and $B^2$ are independent standard two-sided Brownian motions. Then $\e[|W_{x,y}-W_{z,w}|^2]=4CNt(|x-z|+|y-w|)\ge d_H((x,y),(z,w))$. Applying the same inequality to $H-H(0,0)$ and $H(0,0)-H$ gives
    \begin{align*}
        \e\left[\sup_{x,y\in K,\,x\ne y}|H(x,y)|\right]
        &\le\e|H(0,0)|
        +\e\left[\sup_{x,y\in [-N,N]}(H(x,y)-H(0,0))\right]\\
        &\quad+\e\left[\sup_{x,y\in [-N,N]}(H(0,0)-H(x,y))\right]\\
        &\le\e|H(0,0)|
        +4\sqrt{CNt}\,\e\left[\sup_{|u|,|v|\le N}(B^1_u+B^2_v)\right]<\infty,
    \end{align*}
    since $\e[H(0,0)^2]\le Ct$ by~\ref{item:A2}. Thus,~\eqref{expec_claim} holds. Together with its Gaussian independent increments, this completes the proof by~\cite[Theorem~4.5.1(ii)]{Ku90}.
\end{proof}

Moreover, by \cite[Theorem 4.2.5]{Ku90}, there exists a constant $K$ such that 
\begin{align*}
|\e[\cS_{r,x}(a)-a]|&\le K(1+|a|)|x-r|,\\    \left|\e[(\cS_{r,x}(a)-a)(\cS_{r,x}(a')-a')]\right|&\le K(1+|a|)(1+|a'|)|x-r|,
\end{align*}
for all $a,a'\in \r$ and all $r,x$ in a finite time interval. We now derive the random infinitesimal generator of the backward Brownian flow under the regularity conditions.

\begin{proposition}\label{thm:Brownian_flow_back}
Under Assumptions~\ref{assu:H} and~\ref{assu:A}, let $\cS$ be the forward Brownian flow given in Proposition~\ref{thm:Brownian_flow_for}. Then the associated backward flow $\cS^*$ is governed by~\eqref{eq:back_flow}.
\end{proposition}

\begin{proof}
By \cite[Theorem 4.2.10]{Ku90}, the correction term of $F$ is 
$$
C(x,t)=\frac{1}{2} \int_r^t \partial_x \!\left.\left[ \int_{\r} \sigma(x,u,s) \sigma(y,u,s) \dd s \right] \right|_{y=x} \dd u= \frac{1}{4} \int_r^t \partial_x \bar\sigma(x,u) \dd u.
$$
\noindent Let $\Delta=\{-x=u_0<\cdots<u_n=-r\}$ be a partition of $[-x,-r]$, and set $v_k =-u_k$. Then
\begin{align*}
    &\int_{-x}^{-r} F(\cS_{-r,u}(a), \hat{\dd} u)= \lim_{|\Delta|\rightarrow 0} \sum_{k=0}^{n-1} \left[F(\cS_{-r,u_{k+1}}(a), u_{k+1}) -F(\cS_{-r,u_{k+1}}(a),u_k) \right]\\
    =& \lim_{|\Delta|\rightarrow 0} \sum_{k=0}^{n-1} \left[ \int_{u_k}^{u_{k+1}}\!\! \int_{\r} \sigma(\cS_{-r,u_{k+1}}(a),u,s) \cW(\dd s, \dd u)+ \int_{u_k}^{u_{k+1}} b(\cS_{-r,u_{k+1}}(a),u)\dd u\right]\\
    =& \lim_{|\Delta|\rightarrow 0} \sum_{k=0}^{n-1} \left[ \int_{v_{k+1}}^{v_k}\! \int_{\r} \sigma(\cS_{-r,-v_{k+1}}(a),-v,s) \cW^*(\dd s, \dd v)+ \int_{v_{k+1}}^{v_k} b(\cS_{-r,-v_{k+1}}(a),-v)\dd v\right]\\
    =&\int_r^x\!\! \int_{\r} \sigma (\cS_{-r,-u}(a),-u,s)\cW^*(\dd s, \dd u) +\int_r^x  b(\cS_{-r,-u}(a),-u) \dd u  .
\end{align*}

\noindent
The limit is taken in probability. The result then follows from \cite[Theorem~4.2.10]{Ku90}.
\end{proof}

\section{Convergence of approximations}\label{sec:appro}
We introduce two additional assumptions on $(\sigma,b)$, which are weakened versions of~\ref{item:1c} and \ref{item:1d}, respectively. For any interval $[T_1,T_2]\subset \r$,
\begin{enumerate}[label=\textnormal{(2.\alph*)}]
\item there is a constant $K_\sigma>0$ such that $\bar{\sigma}(x,t)\le K_\sigma(1+|x|)^2$ for any $x\in \r$ and $t\in [T_1,T_2]$;\label{item:2c}

\item for each $m \geq 1$ there is a nondecreasing concave $r_m$ on $\r_+$ with $\int_{0_+}\! r_m(z)^{-1} \dd z=\infty$, such that for all $x, y \in [-m,m]$ and $t\in [T_1,T_2]$, $\left|b(x,t)-b(y,t)\right| \leq r_m(|x-y|)$. \label{item:2d}
\end{enumerate}

\noindent 
Therefore, under Assumption~\ref{assu:H} and conditions \ref{item:1a}, \ref{item:1b}, \ref{item:2c} and \ref{item:2d}, the SDE~\eqref{eq:SDE_gener} admits a unique strong solution, and pathwise uniqueness holds (see Appendix~\ref{app:path_unique}). 

This section presents two convergence results concerning approximations of the coefficients and the initial conditions, formulated in Propositions~\ref{thm:conv_appro_1} and~\ref{thm:conv_appro_2}, respectively. As an application of Proposition~\ref{thm:conv_appro_1}, we establish the non-crossing property.

\subsection{Convergence under coefficient approximation}\label{subsec:conv_re}

Before stating the convergence result, we introduce the space of the white noise used in the proofs. Following \cite[Chapter~IV, Example~4]{Walsh1986AnIT}, the white noise $\cW$ takes values locally in the Sobolev space of negative index $H^{-d}_{\mathrm{loc}}(\r^2)$ for $d>1$, which is a Polish space. Hence, by \cite[Theorem~1.3]{Bil99}, the law of $\cW$ is tight. 
We also work with the Polish space $C([r,\infty))$ of continuous functions, endowed with the topology of uniform convergence on compact intervals. A metric generating this topology is
\begin{align}\label{metric_def}
	d(f,g):=\sum_{m=1}^\infty 2^{-m}\min\bigg\{1,\sup_{t\in[r,r+m]}|f(t)-g(t)|\bigg\}.
\end{align}

\begin{proposition}\label{thm:conv_appro_1}
Fix $(a,r)\in\r^2$. Let $(\sigma_n,b_n)_{n\ge1}$ and $(\sigma,b)$ satisfy Assumption~\ref{assu:H} and conditions~\ref{item:1a}, \ref{item:1b}, \ref{item:2c}, and~\ref{item:2d}. Let $\cS^n_{r,\cdot}(a)$, $n\ge1$, and $\cS_{r,\cdot}(a)$ denote the unique strong solutions to~\eqref{eq:SDE_gener} corresponding to $(\sigma_n,b_n)$ and $(\sigma,b)$, respectively. Suppose that:
\begin{enumerate}[label=(\roman*)]
\item for any compact set $K\subset\r$ and all $T\ge r$,
\begin{equation}\label{eq:assume_sig_b}
\begin{gathered}
\lim_{n\to\infty}\sup_{r\le t\le T}\sup_{x\in K}\int_{\r}(\sigma_n(x,t,s)-\sigma(x,t,s))^2\,\dd s=0,\\
\lim_{n\to\infty}\sup_{r\le t\le T}\sup_{x\in K}|b_n(x,t)-b(x,t)|=0;
\end{gathered}
\end{equation}

\item for any $T\ge r$, there exist positive constants $K_b'=K_b'(T)$ and $K_\sigma'=K_\sigma'(T)$ such that, for all $x\in\r$, $r\le t\le T$, and $n\ge1$,
\begin{align}\label{eq:assume_lin_bound}
|b_n(x,t)| \le K_b'(1+|x|), \qquad 
\bar\sigma_n(x,t):=\int_{\r}\sigma_n(x,t,s)^2\,\dd s \le K_\sigma'(1+|x|)^2.
\end{align}
\end{enumerate}
Then $\cS^n_{r,\cdot}(a)\to \cS_{r,\cdot}(a)$ in probability in $C([r,\infty))$. In particular, there exists a subsequence of $\{\cS^n_{r,\cdot}(a)\}$ that converges almost surely uniformly on compact intervals to $\cS_{r,\cdot}(a)$.
\end{proposition}

By the martingale representation theorem, the SDE~\eqref{eq:SDE_gener} is equivalent in law to an SDE driven by a Brownian motion. 
We obtain the following lemma.
  
\begin{lemma}[Lemma 1 in \cite{Kaneko1988}]
    Under the assumptions of Proposition~\ref{thm:conv_appro_1}, for any $T \ge r$ and compact set $K \subset \r$, there exists $C = C_{T,K} > 0$ such that
    \begin{align}\label{eq:lm_S}
    	\sup_{a\in K} \e\left[ \sup_{x\le u,v\le y} |\cS_{r,u}(a)-\cS_{r,v}(a)|^4 \right] &\le C |y-x|^2, \qquad \forall r \le x < y \le T,\\
    	\label{eq:lm_tight}
    	\sup_{n\ge 1} \sup_{a\in K} \e\left[ \sup_{x\le u,v\le y} |\cS^n_{r,u}(a)-\cS^n_{r,v}(a)|^4 \right] &\le C |y-x|^2, \qquad \forall r \le x < y \le T.
    \end{align}
\end{lemma}

\begin{proof}[Proof of Proposition~\ref{thm:conv_appro_1}]
    By~\eqref{eq:lm_tight} and Theorems~4.2 and~4.3 in Chapter~I of~\cite{SDE81}, the sequence $$\{(\cS^n_{r,\cdot}(a), \cS_{r,\cdot}(a), \cW)\}_{n=1}^\infty$$ is tight. By the Skorokhod representation theorem (see, e.g., \cite[Theorem~6.7]{Bil99}), there exists a probability space $(\hat\Omega, \hat \cF, \hat \p)$, on which a subsequence (still indexed by $n$) and $\{ \hat{\cS}^n, \cY^n, \cW^n \}_{n=1}^\infty$ are defined such that: the law of $\{ \hat{\cS}^n, \cY^n, \cW^n \}$ is identical to that of $\{\cS^n_{r,\cdot}(a), \cS_{r,\cdot}(a), \cW\}$ for each $n\ge 1$; and $\{ \hat{\cS}^n, \cY^n, \cW^n \}$  converges to $\{ \hat{\cS}, \cY, \hat{\cW} \}$, where $\hat{\cS}^n \to \hat{\cS}$ and $\cY^n \to \cY$ in $C([r,\infty))$ almost surely, and $\cW^n \to \hat{\cW}$ in $H^{-d}_{\mathrm{loc}}(\r^2)$ almost surely. Then, by~\eqref{eq:lm_S} and~\eqref{eq:lm_tight}, there exists a constant $C=C_T>0$ such that
    \begin{multline*}
    	\qquad \sup_n \hat{\e} \left[ \sup_{r\le x\le T} |\hat{\cS}^n(x)-\cY^n(x)|^4 \right]=\sup_n \e \left[ \sup_{r\le x\le T} |\cS_{r,x}^n(a)-\cS_{r,x}(a)|^4 \right]\\
    	\le 8 \sup_n\e \left[ \sup_{r\le x\le T} |\cS_{r,x}^n(a)-a|^4 \right]+8\, \e \left[ \sup_{r\le x\le T} |\cS_{r,x}(a)-a|^4 \right]
    	\le 16C(T-r)^2,\qquad
    \end{multline*}
    where $\hat{\e}$ stands for the expectation under $\hat \p$. This ensures uniform integrability. Thus, 
    \begin{align}\label{eq:limsup_L2}
        \nonumber \limsup_{n \rightarrow \infty}\e \left[ \sup_{r\le x\le T} |\cS_{r,x}^n(a)-\cS_{r,x}(a)|^2 \right] =&\; \limsup_{n \rightarrow \infty}\hat{\e} \left[ \sup_{r\le x\le T} |\hat{\cS}^n(x)-\cY^n(x)|^2 \right]\\
        =&\; \hat{\e} \left[ \sup_{r\le x\le T} |\hat{\cS}(x)-\cY(x)|^2 \right].
    \end{align}
     
    Each $\hat{\cS}^n$ and $\cY^n$, $n\ge 1$, satisfies, for $x\ge r$,
    \begin{align*}
        \hat{\cS}^n(x) &= a +\int_r^x\!\! \int_{\r} \sigma_n(\hat{\cS}^n(u), u, s) \cW^n(\dd s, \dd u)+ \int_r^x b_n(\hat{\cS}^n(u),u)  \dd u,\\
        \cY^n(x) &= a +\int_r^x\!\! \int_{\r} \sigma(\cY^n(u), u, s) \cW^n(\dd s, \dd u)+ \int_r^x b(\cY^n(u), u)  \dd u.
    \end{align*}

	\noindent
	Letting $n \rightarrow \infty$ and using Proposition~\ref{lm:conv_sigma}, we deduce that both $\hat{\cS}$ and $\cY$ solve
	$$
	X(x) = a +\int_r^x\!\! \int_{\r} \sigma(X(u), u, s) \hat{\cW}(\dd s, \dd u)+ \int_r^x b(X(u),u) \dd u, \quad x \ge r .
	$$
	By pathwise uniqueness, $\hat{\cS}=\cY$ almost surely, and by~\eqref{eq:limsup_L2}, we obtain
	$$
	\lim_{n \rightarrow \infty} \e \left[ \sup_{r\le x\le T} |\cS_{r,x}^n(a)-\cS_{r,x}(a)|^2 \right]=0.
	$$
	By~\eqref{metric_def}, $\cS_{r,\cdot}^n(a)$ converges in probability to $\cS_{r,\cdot}(a)$ as $n \rightarrow \infty$ in $C([r,\infty))$.
\end{proof}

\subsection{Convergence under initial value approximation}

Assume that the coefficients $(\sigma,b)$ satisfy Assumption~\ref{assu:H} and conditions~\ref{item:1a}--\ref{item:1d}, and let $\cS$ and $\cS^{\mathrm b}$ denote the associated forward and backward solutions to~\eqref{eq:SDE_gener} and~\eqref{eq:back_flow}, respectively. For each starting point $(a,r)\in\r^2$, we glue their trajectories by
\begin{align}\label{def:overline_S}
    \overline{\cS}_{r,x}(a):= 
    \begin{cases}
\cS_{r,x}(a), & \text{if }\  x \ge r, \\
\cS^{\mathrm{b}}_{r,x}(a), & \text{if }\  x \le r.
\end{cases}
\end{align}

Let $(\sigma_n,b_n)$ satisfy Assumptions~\ref{assu:H} and~\ref{assu:A}. Let $\cS^n$ denote the corresponding Brownian flow as in Proposition~\ref{thm:Brownian_flow_for}. By Proposition~\ref{thm:Brownian_flow_back}, the associated backward flow is governed by~\eqref{eq:back_flow} with coefficients $(\sigma_n,b_n)$.
Suppose further that, for any compact set $K \subset \r$ and all $T \ge r$, the pair $(\sigma_n, b_n)$ satisfies~\eqref{eq:assume_sig_b} and
\begin{align}\label{eq:assume_sig_partial}
    \lim_{n \rightarrow \infty} \sup_{r\le t\le T} \sup_{x\in K} 
    \left| \partial_x \bar\sigma_n(x,t) - \partial_x \bar\sigma(x,t) \right| = 0,
\end{align}
and that there exist constants $K_b', K_\sigma'>0$ such that, for all $x\in \r$, $r\le t\le T$, and $n\ge 1$,
\begin{align}\label{eq:assume_lin_bound_1}
    |b_n(x,t)|\le K_b'(1+|x|),\qquad 
    |\partial_x \bar{\sigma}_n(x,t)| \le K_\sigma'(1+|x|).
\end{align}
Then, for each $(a,r)\in \r^2$, the backward process $(\cS^n_{r,x}(a),\, x\le r)$ converges, in the sense of Proposition~\ref{thm:conv_appro_1}, to $(\cS^{\mathrm{b}}_{r,x}(a),\, x\le r)$. 
The following result describes the non-crossing property of the curves $\overline{\cS}$ with countably many starting points.

\begin{proposition}\label{prop:non-cros}
    Suppose that $(\sigma,b)$ satisfy Assumption~\ref{assu:H} and conditions~\ref{item:1a}--\ref{item:1d}. Then, almost surely, for all $(a_1,r_1),(a_2,r_2)\in \bD^2$ and all $y,z\in \r$,
    \begin{align*}
        \overline\cS_{r_1,y}(a_1)<  \overline\cS_{r_2,y}(a_2) \ \Longrightarrow \ \overline\cS_{r_1,z}(a_1)\le \overline\cS_{r_2,z}(a_2).
    \end{align*}
\end{proposition}

\begin{proof}
    By Lemma~\ref{lm:C1}, we construct approximations $(\sigma_n,b_n)$ satisfying Assumptions~\ref{assu:H} and~\ref{assu:A}, together with~\eqref{eq:assume_sig_b}, \eqref{eq:assume_sig_partial}, and~\eqref{eq:assume_lin_bound_1}.
    Then, by Proposition~\ref{thm:conv_appro_1}, with probability one there exists a subsequence $\{n_k\}$ such that, for each $(a,r)\in \bD^2$, $(\cS^{n_k}_{r,x}(a), \, x\ge r)$ and $(\cS^{n_k}_{r,x}(a), \, x\le r)$ converge to $(\cS_{r,x}(a), \, x\ge r)$ and $(\cS^{\mathrm{b}}_{r,x}(a), \, x\le r)$ uniformly on compact intervals, respectively. For each $n_k$, $\cS^{n_k}$ is a Brownian flow on the full-trajectory and hence satisfies the strict non-crossing property: 
    \begin{align*}
        \cS^{n_k}_{r_1,y}(a_1)< \cS^{n_k}_{r_2,y}(a_2) \; \Longrightarrow \; \cS^{n_k}_{r_1,z}(a_1)< \cS^{n_k}_{r_2,z}(a_2).
    \end{align*}
    Passing to the limit along the subsequence converts the strict inequality into a non-strict one, which establishes the result for all $(a_1,r_1),(a_2,r_2)\in \bD^2$.
\end{proof}

The following result follows from the same proof, using Lemma~\ref{lm:a1} instead.

\begin{proposition}\label{cor:non_cross}
    Suppose that $(\sigma,b)$ satisfy Assumption~\ref{assu:H} and conditions \ref{item:1a}, \ref{item:1b}, \ref{item:2c}, and \ref{item:2d}. Then, almost surely, for all $(a_1,r_1), (a_2,r_2) \in \bD^2$ and all $z \ge y \ge r_1\vee r_2$,
    $$
        \cS_{r_1,y}(a_1) < \cS_{r_2,y}(a_2) 
        \; \Longrightarrow\; 
        \cS_{r_1,z}(a_1) \le \cS_{r_2,z}(a_2).
    $$
\end{proposition}

Finally, we state a convergence result with respect to the initial value.

\begin{proposition}\label{thm:conv_appro_2}
    Fix $(a,r)\in\r^2$, and let $(a_n)$ be a deterministic sequence with $a_n\downarrow a$. Suppose that $(\sigma,b)$ satisfy Assumption~\ref{assu:H} and conditions~\ref{item:1a}, \ref{item:1b}, \ref{item:2c}, and~\ref{item:2d}. Let $\cS_{r,\cdot}(a_n)$, $n\ge1$, and $\cS_{r,\cdot}(a)$ denote the unique strong solutions to~\eqref{eq:SDE_gener} with initial values $a_n$ and $a$. Then, almost surely, 
    $$
    \lim_{a_n\downarrow a}\cS_{r,x}(a_n)=\cS_{r,x}(a),
    $$
    uniformly on compact intervals.
\end{proposition}

\begin{proof}
By~\eqref{eq:lm_tight} and Theorems~4.2 and~4.3 in \cite[Chapter~I]{SDE81}, the sequence $\{\cS_{r,\cdot}(a_n)\}_{n\ge 1}$ is tight. Following the argument in the proof of Proposition~\ref{thm:conv_appro_1}, there exists a subsequence converging to $\cS_{r,\cdot}(a)$ in $C([r,\infty))$ almost surely. The proof of Proposition~\ref{cor:non_cross} also yields the non-crossing property simultaneously for the starting points in the countable set $\{(a,r)\}\cup\{(a_n,r):n\ge1\}$, which completes the proof.
\end{proof}

\section{Proof of the main theorems}\label{sec:exist}
This section proves the main results. In Section~\ref{subsec:exist}, we establish the existence and uniqueness of the stochastic flow. Section~\ref{subsec:dual} then studies the SDE formulation of the dual flow, leading to a proof of Theorem~\ref{thm:backward_flow_SDE}.

\subsection{Existence and uniqueness}\label{subsec:exist}

For any $(a,r)\in \bD^2$, let $\cS_{r,\cdot}(a)$ denote the almost sure unique strong solution to~\eqref{eq:SDE_gener}, and for each $x\in \r$, define
\begin{align}\label{eq:def_M}
    M(x):=\{\cS_{r,x}(a):\, r<x,\, (a,r)\in \bD^2\}.
\end{align}
The following two lemmas serve as auxiliary results, providing the necessary tools for establishing the existence and uniqueness of the stochastic flow.

\begin{lemma}\label{prop:density}
    Suppose that $(\sigma,b)$ satisfy Assumption~\ref{assu:H} and conditions \ref{item:1a}, \ref{item:1b}, \ref{item:2c}, and \ref{item:2d}. Then, almost surely, for any $x\in \r$, the set $M(x)$ is dense in $\r$.
\end{lemma}

\begin{proof}
    It suffices to show that for any fixed dyadic numbers $h,K>0$ and $a$, almost surely, $M(x)\cap (a-h,a+h)\neq \varnothing$ for all $x\in (-K,K]$. Let $k\in \mathbb{N}_+$ such that $t:=K2^{-k}$. For every $y\in(-K,K]$ there is an integer $-2^k\le j<2^k$ with $y\in(jt,(j+1)t]$. Hence,
    \begin{multline*}
        \p \bigl(\exists y\in(-K,K]:\, M(y)\cap(a-h,a+h)=\varnothing \bigr) \le  \sum_{j=-2^k}^{2^k-1}\p \bigg(\sup_{y\in[jt,(j+1)t]}|\cS_{jt,y}(a)-a|\ge h\bigg)\\
        \le \sum_{j=-2^k}^{2^k-1} \frac{1}{h^4} \e \bigg[\sup_{y\in[jt,(j+1)t]}|\cS_{jt,y}(a)-a|^4\bigg]\le  2h^{-4}CKt \;\longrightarrow\; 0, \qquad \text{ as }\  t\rightarrow 0,
    \end{multline*}
    where the last inequality follows from~\eqref{eq:lm_S}. This completes the proof.
\end{proof}

\begin{lemma}\label{prop:density2}
    Fix $(a,r)\in \r^2$, and let $\cS_{r,\cdot}(a)$ denote the unique strong solution to~\eqref{eq:SDE_gener}. Under the assumptions in Lemma~\ref{prop:density}, almost surely,  
    \begin{align}\label{eq:cS_D}
        \cS_{r,x}(a) = \inf \left\{ \cS_{r_n,x}(a_n):\, r_n< r,\, \cS_{r_n,r}(a_n)>a,\, (a_n,r_n)\in \bD^2 \right\}.
    \end{align}
\end{lemma}

\begin{proof}
    Replacing $\bD^2$ by $\bD^2\cup \{(a_n,r):\, a_n\in \bD\cup \{a\} \}$ in Proposition~\ref{cor:non_cross}, the non-crossing property remains valid. It follows that $\cS_{r,x}(a)$ is bounded above by the right-hand side of~\eqref{eq:cS_D}. 
    For any $\bD\ni a_n>a$, since $M(r)$ is dense, there exist $(a_n',r_n')\in \bD^2$ with $r_n'<r$ such that $\cS_{r_n',r}(a_n')\in (a,a_n)$. Thus, $\cS_{r_n',x}(a_n')\le \cS_{r,x}(a_n)$ by the non-crossing property. Letting $a_n\downarrow a$, by Proposition~\ref{thm:conv_appro_2}, we conclude that the right-hand side of~\eqref{eq:cS_D} is bounded above by $\cS_{r,x}(a)$, completing the proof.
\end{proof}

The next result provides the existence and uniqueness of the stochastic flow.

\begin{theorem}\label{thm:ex}
    Suppose that $(\sigma,b)$ satisfy Assumption~\ref{assu:H} and conditions \ref{item:1a}, \ref{item:1b}, \ref{item:2c}, and \ref{item:2d}. Then the associated stochastic flow $\cS$ defined in Definition~\ref{def:flow} exists and is unique.
\end{theorem}

\begin{proof}
    For any $(a,r)\in \bD^2$, let $\cS_{r,\cdot}(a)$ denote the unique strong solution to~\eqref{eq:SDE_gener}. For any general $(a,r)\in \r^2$, we define the flow line $\cS_{r,\cdot}(a)$ via~\eqref{eq:cS_D}. Lemmas~\ref{prop:density} and~\ref{prop:density2} guarantee the consistency and well-definedness of this construction.

    \smallskip
    \textbf{Existence.} By Lemma~\ref{prop:density2}, for any $(a,r)\in \r^2$, the strong solution to~\eqref{eq:SDE_gener} necessarily satisfies~\eqref{eq:cS_D} almost surely. Hence, $\cS_{r,\cdot}(a)$ is the almost sure unique strong solution of~\eqref{eq:SDE_gener}. 
    Property~\ref{item:R1} for $\cS$ follows directly from Lemma~\ref{prop:density}.
    By~\eqref{eq:cS_D}, for any $r\le x$, the map $a\mapsto \cS_{r,x}(a)$ is nondecreasing and right-continuous, which implies property~\ref{item:R2}.
    To verify property~\ref{item:R3}, let $(a_1,r_1), (a_2,r_2)\in \r^2$ and $y\ge \max\{r_1,r_2\}$ with $\cS_{r_1,y}(a_1)<\cS_{r_2,y}(a_2)$. By~\eqref{eq:cS_D}, there exists $(a',r')\in \bD^2$ with $r'<r_1$ such that $\cS_{r',r_1}(a')>a_1$ and $\cS_{r',y}(a')\in [\cS_{r_1,y}(a_1), \cS_{r_2,y}(a_2))$. For any $(a_n',r_n')\in \bD^2$ with $r_n'<r_2$ and $\cS_{r_n',r_2}(a_n')>a_2$, we have $\cS_{r_n',y}(a_n')\ge \cS_{r_2,y}(a_2)>\cS_{r',y}(a')$. It follows from Proposition~\ref{cor:non_cross} and~\eqref{eq:cS_D} that $\cS_{r_n', z}(a_n')\ge \cS_{r', z}(a')\ge \cS_{r_1,z}(a_1)$ for any $z\ge y$. Taking the infimum over all such $(a_n',r_n')\in \bD^2$ then yields property~\ref{item:R3}. Therefore, the family $\{(\cS_{r,x}(a),\, x\ge r)\}_{(a,r)\in\r^2}$ constructed above indeed defines a stochastic flow.

    \smallskip
    \textbf{Uniqueness.} Suppose that $\tcS$ is another stochastic flow as in Definition~\ref{def:flow}. Then, $(\tcS_{r,x}(a),\, x \ge r)_{(a,r)\in \bD^2}$ and $(\cS_{r,x}(a),\, x \ge r)_{(a,r)\in \bD^2}$ coincide almost surely, since both are given by the unique strong solutions to~\eqref{eq:SDE_gener}. It therefore suffices to verify that $\tcS$ satisfies~\eqref{eq:cS_D}. By property~\ref{item:R3}, $\tcS_{r,x}(a)$ is almost surely bounded above by the right-hand side of~\eqref{eq:cS_D}. Using property~\ref{item:R2}, almost surely, for any $x\ge r$, 
    $$\lim_{a_n\downarrow a}\tcS_{r,x}(a_n)=\tcS_{r,x}(a).$$
    By Lemma~\ref{prop:density}, we may choose $(a_n',r_n')\in \bD^2$ with $r_n'<r$ such that $\tcS_{r_n',r}(a_n')\in (a,a_n)$. Property~\ref{item:R3} then yields $\tcS_{r_n',x}(a_n')\le \tcS_{r,x}(a_n)$. Letting $a_n \downarrow a$, we conclude that the right-hand side of~\eqref{eq:cS_D} is almost surely bounded above by $\tcS_{r,x}(a)$. Hence, $\tcS_{r,x}(a)$ is uniquely determined by~\eqref{eq:cS_D}, and therefore $\tcS = \cS$ almost surely.
\end{proof}

\begin{corollary}\label{cor:ex}
    Under the assumptions of Theorem~\ref{thm:ex}, let $\cS$ be the associated stochastic flow. Then, almost surely, for all $(a,r)\in\r^2$ and $x\ge r$, $\cS_{r,x}(a)$ satisfies~\eqref{eq:cS_D}.
\end{corollary}

\subsection{Dual flow}\label{subsec:dual}
By Proposition~\ref{thm:Brownian_flow_back}, the backward SDE can be derived from the forward one under sufficient regularity. The goal of this section is to extend this result to coefficients $(\sigma,b)$ that lack the required regularity but arise as limits of $(\sigma_n,b_n)$.

Under the assumptions of Theorem~\ref{thm:backward_flow_SDE}, Theorem~\ref{thm:ex} guarantees the existence and uniqueness of the stochastic flows $\cS$ and $\cS^{\mathrm{b}}$, whose trajectories $(\cS_{r,x}(a),\, x\ge r)$ and $(\cS^{\mathrm{b}}_{-r,-x}(a),\, x\ge r)$ are almost surely the strong solutions to~\eqref{eq:SDE_gener} and~\eqref{eq:back_flow}, respectively. Applying Corollary~\ref{cor:ex} to $\cS^{\mathrm{b}}$, we obtain that, almost surely, for any $(b,r)\in \r^2$,
\begin{align}\label{eq:cS_b}
    \cS^{\mathrm{b}}_{r,x}(b)
    = \inf \left\{ \cS^{\mathrm{b}}_{r_n,x}(b_n):\, r_n> r,\, \cS^{\mathrm{b}}_{r_n,r}(b_n)>b,\, (b_n,r_n)\in \bD^2 \right\}.
\end{align}

\begin{proof}[Proof of Theorem~\ref{thm:backward_flow_SDE}]
It suffices to show that the backward flow $\cS^{\mathrm{b}}$ coincides with $\cS^*$ as defined in~\eqref{eq:def_dual}, that is, for any $(b,r)\in \r^2$ and $x\le r$,
\begin{align}\label{eq:back_SDE_dual}
    \cS^{{\mathrm{b}}}_{r, x}(b)=\inf \{a \in \r:\, \cS_{x,r}(a)>b \}.
\end{align}

Suppose $a \in \r$ satisfies $\cS_{x,r}(a) \le b$. For any $(b_n,r_n)\in \bD^2$ with $r_n > r$ and $\cS^{\mathrm{b}}_{r_n,r}(b_n) > b $,~\eqref{eq:cS_D} gives $(a_n,x_n)\in \bD^2$ with $x_n<x$, $\cS_{x_n,x}(a_n)>a$ and $\cS_{x_n,r}(a_n)\in [\cS_{x,r}(a), \cS^{\mathrm{b}}_{r_n,r}(b_n)) $. Proposition~\ref{prop:non-cros} then implies $\cS^{\mathrm{b}}_{r_n,x}(b_n) \ge \cS_{x_n,x}(a_n)>a$, and~\eqref{eq:cS_b} yields that $\cS^{\mathrm{b}}_{r,x}(b) \ge a$.

Conversely, suppose $a \in \r$ satisfies $\cS_{x,r}(a) > b$. By Lemma~\ref{prop:density} applied to $\cS^{\mathrm{b}}$, there exist $(b_n,r_n)\in \bD^2$ with $r_n > r$ such that $\cS^{\mathrm{b}}_{r_n,r}(b_n)\in (b,\cS_{x,r}(a))$.
For any $(a_n,x_n)\in \bD^2$ with $x_n<x$ and $\cS_{x_n,x}(a_n)>a$,~\eqref{eq:cS_D} gives $\cS_{x_n,r}(a_n)\ge \cS_{x,r}(a)>\cS^{\mathrm{b}}_{r_n,r}(b_n)$. Applying Proposition~\ref{prop:non-cros}, we obtain $\cS_{x_n,x}(a_n)\ge \cS^{\mathrm{b}}_{r_n,x}(b_n)$, and taking the infimum yields $a\ge \cS^{\mathrm{b}}_{r,x}(b)$. This completes the proof of~\eqref{eq:back_SDE_dual}. 
\end{proof}

\subsection{Non-crossing property}\label{sec:non-cros}

Recall that $\overline{\cS}$ is obtained by gluing the flow lines of $\cS$ and $\cS^{\mathrm b}$ at their starting points as in~\eqref{def:overline_S}. By the proof of Theorem~\ref{thm:backward_flow_SDE}, this is equivalent to gluing $\cS$ with its dual flow $\cS^*$. In this section, we establish a non-crossing property for the two-sided flow $\overline{\cS}$, extending Proposition~\ref{prop:non-cros} to the whole plane.

\begin{lemma}
Almost surely, the following properties hold:
\begin{enumerate}[label=(\roman*)]
    \item For any fixed $a,r,x\in\r$ with $x\ge r$,
    \begin{align}\label{eq:flow_prop}
        \cS_{r,y}(a)=\cS_{x,y}\circ\cS_{r,x}(a),\qquad\forall y\ge x.
    \end{align}

    \item For any $(a,r)\in\r^2$ and $y\ge x\ge r$,
    \begin{align}\label{eq:one_per_flow}
        \cS_{r,y}(a)\le \cS_{x,y}\circ\cS_{r,x}(a).
    \end{align}
\end{enumerate}
\end{lemma}

\begin{proof}
    Since $(\cS_{r,y}(a),\, y\ge x)$ is a strong solution to~\eqref{eq:SDE_gener} started from the (random) initial value $\cS_{r,x}(a)$ at time $x$ (see the remarks in \cite[pp.~369]{RevYor}), pathwise uniqueness implies that conditional on $\cS_{r,x}(a) = b$, $(\cS_{r,y}(a),\, y\ge x)$ and $(\cS_{x,y}(b),\, y\ge x)$ coincide almost surely. Integrating with respect to the law of $\cS_{r,x}(a)$ shows that the probability that~\eqref{eq:flow_prop} holds is one. For any $a'>\cS_{r,x}(a)$, properties~\ref{item:R1} and~\ref{item:R3} imply that $\cS_{x,y}(a')\ge \cS_{r,y}(a)$ for all $y \ge x$. Taking the infimum over all such $a'$ and using property~\ref{item:R2}, we obtain~\eqref{eq:one_per_flow}.
\end{proof}

\begin{proposition}[Non-crossing property]\label{prop:non-crossing}
     Suppose that $(\sigma,b)$ satisfy Assumption~\ref{assu:H} and conditions~\ref{item:1a}--\ref{item:1d}. Then, almost surely, for all $(a_1,r_1),(a_2,r_2)\in \r^2$ and all $y,z\in \r$,
    \begin{align*}
        \overline\cS_{r_1,y}(a_1)< \overline\cS_{r_2,y}(a_2) \ \Longrightarrow \ \overline\cS_{r_1,z}(a_1)\le \overline\cS_{r_2,z}(a_2).
    \end{align*}
\end{proposition}

\begin{proof}
    We argue by contradiction. Suppose there exist parameters such that $\overline\cS_{r_1,y}(a_1)<\overline\cS_{r_2,y}(a_2)$ but $\overline\cS_{r_1,z}(a_1)>\overline\cS_{r_2,z}(a_2)$. Without loss of generality, assume that $r_1 < r_2$. By property~\ref{item:R3} for both $\cS$ and $\cS^{\mathrm b}$, $y\vee z>r_1$ and $y\wedge z<r_2$. We distinguish three cases.
    
    \textbf{Case 1:} $y,z\in [r_1,r_2]$. By assumption, $\cS_{r_1,y}(a_1)<\cS^{\mathrm{b}}_{r_2,y}(a_2)$ and $\cS_{r_1,z}(a_1)>\cS^{\mathrm{b}}_{r_2,z}(a_2)$. By Corollary~\ref{cor:ex}, we can choose $a_1',r_1',a_2',r_2' \in \bD$ with $r_1' < r_1 < r_2 < r_2'$, $\cS_{r_1',r_1}(a_1')>a_1$ and $\cS^{\mathrm{b}}_{r_2',r_2}(a_2')>a_2$ such that 
    $$
    \cS_{r_1',y}(a_1')\in [\cS_{r_1,y}(a_1), \cS^{\mathrm{b}}_{r_2,y}(a_2)), 
    \qquad
    \cS^{\mathrm{b}}_{r_2',z}(a_2')\in [\cS^{\mathrm{b}}_{r_2,z}(a_2), \cS_{r_1,z}(a_1)).
    $$
    Applying Corollary~\ref{cor:ex} again gives
    $$
    \cS_{r_1',z}(a_1')\ge \cS_{r_1,z}(a_1)>\cS^{\mathrm{b}}_{r_2',z}(a_2'), \quad 
    \cS^{\mathrm{b}}_{r_2',y}(a_2')\ge \cS^{\mathrm{b}}_{r_2,y}(a_2)>\cS_{r_1',y}(a_1'),
    $$
    which contradicts Proposition~\ref{prop:non-cros}.
    
    \textbf{Case 2:} $z>r_2$ or $y<r_1$. By symmetry, it suffices to consider $z>r_2$, which gives $\cS_{r_1,z}(a_1) > \cS_{r_2,z}(a_2)$. Property~\ref{item:R3} together with~\eqref{eq:one_per_flow} implies $\cS_{r_1,r_2}(a_1)>a_2$. As established in Case 1, for any $x\in [r_1,r_2]$, $\cS_{r_1,x}(a_1)\ge \cS^{\mathrm{b}}_{r_2,x}(a_2)$. In particular, $a_1\ge \cS^{\mathrm{b}}_{r_2,r_1}(a_2)$, and hence $\cS^{\mathrm{b}}_{r_1,x}(a_1)\ge \cS^{\mathrm{b}}_{r_2,x}(a_2)$ for any $x\le r_1$. This rules out the existence of any $y<r_2$ satisfying the initial assumption.

    \textbf{Case 3:} $z<r_1$ or $y>r_2$. By symmetry, it suffices to consider $y>r_2$, which gives $\cS_{r_1,y}(a_1)<\cS_{r_2,y}(a_2)$. Let $b\in (\cS_{r_1,y}(a_1), \cS_{r_2,y}(a_2))$. By property~\ref{item:R3}, for any $a'>a_2$ we have $\cS_{r_2,y}(a')\ge \cS_{r_2,y}(a_2)>b$, while for any $a'<\cS_{r_1,r_2}(a_1)$ we have $\cS_{r_2,y}(a')\le \cS_{r_1,y}(a_1)<b$. Hence, by~\eqref{eq:back_SDE_dual} and property~\ref{item:R2}, $\cS^{\mathrm{b}}_{y,r_1}(b)>a_1$ and $\cS^{\mathrm{b}}_{y,r_2}(b)\in [\cS_{r_1,r_2}(a_1),a_2]$. It follows from~\eqref{eq:one_per_flow} that $\cS^{\mathrm{b}}_{r_2,r_1}(a_2)\ge \cS^{\mathrm{b}}_{y,r_1}(b)>a_1$. Combining Case 1 with property~\ref{item:R3}, this rules out the existence of any $z<r_2$ satisfying the initial assumption. 
    
    All cases lead to a contradiction, which completes the proof.
\end{proof}

\section{Flows on strip-like domains and their properties}\label{sec:strip}
Let $I\subset\r$ be an interval, possibly unbounded, and open at its right endpoint, and let $\cW$ be a white noise on $I\times\r$. Having studied stochastic flows on the whole plane, we now consider flows on the strip $I\times\r$. Throughout this section, we assume that $(\sigma,b)$ satisfy Assumption~\ref{assu:H} and conditions~\ref{item:1a}, \ref{item:1b}, \ref{item:2c}, and~\ref{item:2d} on $I$.

Let $d_1:=\inf I$ and $d_2:=\sup I$. For every sufficiently large integer $n$, set $I_n:=\bigl(d_1+\tfrac1n,\,d_2-\tfrac1n\bigr)$, with the convention $-\infty+c=-\infty$ and $+\infty-c=+\infty$. For $\eta=(\eta_1,\eta_2)\in\{0,1\}^2$, we say that a solution to~\eqref{eq:SDE_gener} starting from $(a,r)\in I\times\r$ satisfies the boundary conditions $\eta$ if, whenever $\eta_i=0$, it is absorbed upon hitting $d_i$, and the obtained process stays in $\bar I$ for all times. When $|d_i|=\infty$, the boundary condition $\eta_i$ is irrelevant.

\begin{definition}\label{def:strip_flow}
Let $\eta\in \{0,1\}^2$. A \emph{stochastic flow} on $I\times \r$ driven by $\cW$ with boundary conditions $\eta$ is a family of processes 
$$
\cS := \left\{\bigl(\cS_{r,x}(a),\, x \ge r \bigr)\right\}_{(a,r) \in I \times \r}
$$
with the regularity properties \ref{item:R1}--\ref{item:R3} in Definition~\ref{def:flow} (restricted to $I\times \r$) such that for each fixed $(a,r)\in I\times \r$,  $\cS_{r,\cdot}(a)$ is the almost sure unique strong solution to~\eqref{eq:SDE_gener} and satisfies the boundary conditions $\eta$. 
\end{definition}

Let $\cI:=(\bD\cap I)\times\bD$ denote the set of dyadic pairs in $I\times\r$. For each $(a,r)\in\cI$, let $(\cS_{r,x}(a),\,x\ge r)$ denote the almost sure unique strong solution to~\eqref{eq:SDE_gener} satisfying the boundary condition $\eta$. For any $x\in\r$, define
\begin{align*}
    M_I(x):=\bigl\{\cS_{r,x}(a):\, r<x,\ (a,r)\in\cI\bigr\}.
\end{align*}
When $I=\r$, this coincides with the set $M(x)$ defined in~\eqref{eq:def_M}. The next result shows that Lemma~\ref{prop:density} and Corollary~\ref{cor:ex} remain valid in the strip setting.

\begin{lemma}\label{lm:consis}
Let $\eta\in\{0,1\}^2$ and suppose that $(\sigma,b)$ satisfy Assumption~\ref{assu:H} and conditions~\ref{item:1a}, \ref{item:1b}, \ref{item:2c}, and~\ref{item:2d} on $I$. Then the following properties hold:
\begin{enumerate}[label=(\roman*)]
    \item Almost surely, for any $x\in \r$, the set $M_I(x)$ is dense in $I$.
    
    \item Suppose further that the associated stochastic flow $\cS$ on $I\times \r$ with boundary conditions $\eta$ exists. Then, almost surely, for all $(a,r)\in I\times \r$ and $x\ge r$,
    \begin{align}\label{eq:cS_I}
        \cS_{r,x}(a) = \inf \left\{ \cS_{r_n,x}(a_n):\,r_n< r,\, \cS_{r_n,r}(a_n)>a,\, (a_n,r_n)\in \cI \right\}.
    \end{align}
\end{enumerate}
\end{lemma}

\begin{proof}
(i) It suffices to show that for any fixed sufficiently large integers $K,N>0$ and any open interval $J\subset I_N$, almost surely, $M_I(x)\cap J\neq \varnothing$ for all $x\in (-K,K]$. By Remark~\ref{rmk:C2}, there exists a pair $(\sigma_N,b_N)$ satisfying~\eqref{eq:b_n_sig_n} on $I_N$, Assumption~\ref{assu:H}, and conditions~\ref{item:1a}, \ref{item:1b}, \ref{item:2c}, and~\ref{item:2d} on $\r$. Let $\cS^N$ be the stochastic flow on $\r^2$ associated with $(\sigma_N,b_N)$, whose existence follows from Theorem~\ref{thm:ex}. Applying Lemma~\ref{prop:density} to $\cS^N$, almost surely, for any $x\in(-K,K]$, there exist $(a_1,r_1),(a_2,r_2)\in\bD^2$ such that $\inf J<\cS^N_{r_1,x}(a_1)<\cS^N_{r_2,x}(a_2)<\sup J$. By continuity of the paths $\cS^N_{r_i,x}(a_i)$, there exists $(a_0,r_0)\in\cI$ such that
$$
r_1\vee r_2\le r_0<x,\qquad
\cS^N_{r_1,r_0}(a_1)<a_0<\cS^N_{r_2,r_0}(a_2),
$$
and
$$
\cS^N_{r_i,y}(a_i)\in J\subset I_N,\qquad y\in[r_0,x],\ \ i=1,2.
$$
Hence, by property~\ref{item:R3} of $\cS^N$, we have $\cS^N_{r_0,y}(a_0)\in J\subset I_N$ for all $y\in[r_0,x]$. By~\eqref{eq:b_n_sig_n} and pathwise uniqueness, $(\cS^N_{r_0,y}(a_0),\, y\in [r_0,x])$ coincides with $(\cS_{r_0,y}(a_0),\, y\in [r_0,x])$, and therefore $\cS_{r_0,x}(a_0)=\cS^N_{r_0,x}(a_0)\in J$. This proves (i).

\smallskip
(ii) Since $M_I(x)$ is dense in $I$, the infimum in~\eqref{eq:cS_I} is well defined. By property~\ref{item:R3}, the right-hand side of~\eqref{eq:cS_I} is bounded below by $\cS_{r,x}(a)$. On the other hand, by property~\ref{item:R2},
$$
\lim_{a_n\downarrow a}\cS_{r,x}(a_n)=\cS_{r,x}(a),
$$
almost surely. For each $a_n>a$, applying (i) we can find $(b_n,r_n)\in \cI$ with $r_n<r$ such that $\cS_{r_n,r}(b_n)\in (a,a_n)$. Letting $n\to\infty$ shows that the right-hand side of~\eqref{eq:cS_I} is bounded above by $\cS_{r,x}(a)$, completing the proof.
\end{proof}

The uniqueness argument in the proof of Theorem~\ref{thm:ex}, together with Lemma~\ref{lm:consis}(i), yields the following uniqueness result for the strip-like stochastic flow.

\begin{proposition}[Uniqueness of the flow]\label{prop:unique}
    Let $\eta\in\{0,1\}^2$ and suppose that $(\sigma,b)$ satisfy Assumption~\ref{assu:H} and conditions~\ref{item:1a}, \ref{item:1b}, \ref{item:2c}, and~\ref{item:2d} on $I$. If the associated stochastic flow $\cS$ on $I\times\r$ with boundary conditions $\eta$ exists, then it is unique.
\end{proposition}

For each trajectory of the stochastic flow $\cS$ starting from $(a,r)\in \cI$, the path $x\mapsto \cS_{r,x}(a)$ is almost surely continuous. The next proposition extends this to all starting points in $I\times\r$, thereby proving the continuity statement in Theorem~\ref{thm:ex_intro}.

\begin{proposition}[Continuity of the flow]\label{prop:conti}
    Under the assumptions of Proposition~\ref{prop:unique}, almost surely, the map $x\mapsto \cS_{r,x}(a)$ is continuous for every $(a,r)\in I\times\r$.
\end{proposition}

\begin{proof}
    By~\eqref{eq:cS_I}, $\cS_{r,\cdot}(a)$ has values in $\bar I$ and is upper semi-continuous as an infimum of continuous functions. To prove continuity, it suffices to rule out upward jumps. 
    Suppose that there exists $(a,r)\in I\times \r$ and a sequence $\{x_k\}$ converging to $x\ge r$ such that $\lim_{k\to \infty} \cS_{r,x_k}(a) < \cS_{r,x}(a) - \varepsilon$ for some $\varepsilon > 0$.
    Then, by Lemma~\ref{lm:consis}(i), one can find $(a', r') \in \cI$ with $r' < x$ such that $\lim_{k\to \infty} \cS_{r,x_k}(a) < \cS_{r',x}(a') < \cS_{r,x}(a)$. Since $x \mapsto \cS_{r',x}(a')$ is continuous, for sufficiently large $k$, we have $\cS_{r,x_k}(a) < \cS_{r',x_k}(a')$. This contradicts property~\ref{item:R3}.
\end{proof}

Let $\cS$ be a stochastic flow on $I\times\r$ with boundary conditions $\eta$. With the convention that $\inf\emptyset=\infty$, the associated backward flow is defined by
\begin{align}\label{eq:def_dual_+}
    \cS_{r, x}^*(b):=\inf \{a \in I:\, \cS_{x,r}(a)>b \}\wedge \sup I, \qquad b\in I,\, x\le r.
\end{align}
Then, for any $(a,r)\in I\times \r$ and $x\ge r$, $(\cS^*)^*_{r,x}(a)=\cS_{r,x}(a)$. Indeed, since $\cS_{x,r}^*(b)>a$ implies that $\cS_{r,x}(a)\le b$, we have
\begin{align*}
    (\cS^*)^*_{r,x}(a)&=\inf \{b \in I:\, \cS_{x,r}^*(b)>a \}\wedge \sup I\\
    &\ge \inf \{b \in I:\, \cS_{r,x}(a)\le b \}\wedge \sup I=\cS_{r,x}(a).
\end{align*}
To prove the reverse inequality, first note that if $\cS_{r,x}(a)=\sup I$, then $(\cS^*)^*_{r,x}(a)\le \cS_{r,x}(a)$ holds. If instead $\cS_{r,x}(a)<\sup I$, let $\delta>0$ be sufficiently small so that $b:=\cS_{r,x}(a)+\delta\in I$. By property~\ref{item:R2}, $\cS^*_{x,r}(b)=\inf \{c \in I:\, \cS_{r,x}(c)>\cS_{r,x}(a)+\delta \}\wedge \sup I>a$. Hence, we have $(\cS^*)^*_{r,x}(a)\le \cS_{r,x}(a)+\delta$. Letting $\delta\downarrow0$ yields $(\cS^*)^*_{r,x}(a)\le \cS_{r,x}(a)$. This demonstrates that $\cS$ and $\cS^*$ are mutually dual. That is,
$$
\cS_{r, x}(a)=\inf \{b \in I:\, \cS_{x,r}^*(b)>a \}\wedge \sup I, \qquad a\in I,\, x\ge r.
$$

Let $(\sigma,b)$ and $(\sigma^*,b^*)$ be two pairs of coefficients, and let $\eta=(\eta_1,\eta_2)\in\{0,1\}^2$ and $\eta^*:=(1-\eta_1,1-\eta_2)$. Let $\cS$ and $\cS^{\mathrm b}$ denote the stochastic flows on $I\times\r$ associated with $(\sigma,b)$ and $(\sigma^*,b^*)$, respectively, with boundary conditions $\eta$ and $\eta^*$, where $\cS$ is governed by~\eqref{eq:SDE_gener} and $\cS^{\mathrm b}$ by
\begin{align}\label{sde:back}
\cS^{\mathrm b}_{-r,-x}(a)=a+\int_r^x\!\!\int_{\r}
\sigma^*\bigl(\cS^{\mathrm b}_{-r,-u}(a),-u,s\bigr)\cW^*(\dd s,\dd u)+\int_r^x b^*\bigl(\cS^{\mathrm b}_{-r,-u}(a),-u\bigr)\dd u.
\end{align}
We say that $(\cS,\cS^{\mathrm b})$ satisfy the \emph{absorption condition} if, for every fixed $r_1<r_2$ and $a_1,a_2\in I$, the following holds almost surely. If one of the trajectories $\cS_{r_1,\cdot}(a_1)$ (when $\eta_i=0$) and $\cS^{\mathrm b}_{r_2,\cdot}(a_2)$ (when $\eta_i^*=0$) is absorbed at $d_i$ at some level $s\in[r_1,r_2]$, then the other trajectory does not hit $d_i$ at the same level.

\begin{theorem}\label{thm:dual_strip}
Let both $(\sigma,b)$ and $(\sigma^*,b^*)$ satisfy Assumption~\ref{assu:H} and conditions~\ref{item:1a}, \ref{item:1b}, \ref{item:2c}, and~\ref{item:2d} on $I$. Assume that for all $x\in I$ and $s,t\in\r$,
\begin{align}\label{eq:rela}
\sigma^*(x,t,s)=-\sigma(x,t,s),\qquad
b^*(x,t)=-b(x,t)+\tfrac12\partial_x\bar\sigma(x,t).
\end{align}
Assume that the flows $(\cS,\cS^{\mathrm b})$ defined above exist and satisfy the absorption condition.
Then, almost surely, $\cS^{\mathrm b}=\cS^*$ in the sense of~\eqref{eq:def_dual_+}. In particular, $\{(\cS^*_{-r,-x}(a),\,x\ge r)\}_{(a,r)\in I \times \r}$ forms a stochastic flow on $I\times\r$ with boundary conditions $\eta^*$, and $(\cS^*)^*=\cS$.
\end{theorem}

\begin{remark}\label{rmk:sufficient_condi}
    Suppose that $\sigma$ and $b$ are time-homogeneous (i.e., $\sigma(x,t,s)\equiv\sigma(x,s)$ and $b(x,t)\equiv b(x)$). The following conditions are sufficient for the absorption condition in Theorem~\ref{thm:dual_strip}: (i) for every fixed $(a,r)\in I\times\r$ and $x>r>y$, $\cS^{\mathrm b}_{r,y}(a)$ (when $\eta_i=0$) and $\cS_{r,x}(a)$ (when $\eta_i^*=0$) have no atoms at $d_i$; (ii) for each $d_i\in I$, the first hitting time of $d_i$ by $\cS$ (when $\eta_i=0$) or by $\cS^{\mathrm b}$ (when $\eta_i^*=0$) has an atomless distribution on $(0,\infty)$. Indeed, suppose that $\eta_i=0$, and fix $r_1<r_2$ and $a_1,a_2\in I$. Set
    $$
    \tau:=\inf\{t\ge r_1:\,\cS_{r_1,t}(a_1)=d_i\}.
    $$
    It suffices to show that
    $$
    \p\!\left(\tau\le r_2,\, \cS^{\mathrm b}_{r_2,\tau}(a_2)=d_i\right)=0.
    $$
    For $t\ge r_1$, let $\cF_t:=\sigma(\cW|_{I\times [r_1,t]})$ be the natural right-continuous filtration of the white noise. Since $\tau$ is an $\cF_t$-stopping time and $\{\tau\le r_2\}\in\cF_\tau$, conditioning on $\cF_\tau$ yields
    \begin{align*}
        \p\!\left(\tau\le r_2, \cS^{\mathrm b}_{r_2,\tau}(a_2)=d_i\right) = \e\!\left[\mathbbm{1}_{\{\tau\le r_2\}}\p\!\left(\cS^{\mathrm b}_{r_2,\tau}(a_2)=d_i \mid \cF_\tau \right)\right].
    \end{align*}
    By strong uniqueness for~\eqref{sde:back}, there exists a measurable map $\Psi$ such that $\cS^{\mathrm b}_{r_2,\tau}(a_2) = \Psi\big(\tau, \cW|_{I\times [\tau, r_2]}\big)$ almost surely on $\{\tau\le r_2\}$. The no-atom assumption (i) implies $g(t):=\p(\cS^{\mathrm b}_{r_2,t}(a_2)=d_i)=0$ for all $t\in[r_1,r_2)$. Since the shifted post-$\tau$ noise is independent of $\cF_\tau$, 
    $$
    \p\!\left(\cS^{\mathrm b}_{r_2,\tau}(a_2)=d_i \mid \cF_\tau \right) = g(\tau)=0,\qquad \text{on } \{\tau<r_2\}
    $$
    almost surely. Assumption~(ii) handles $\tau=r_2$, and the case $\eta_i^*=0$ follows by symmetry. Hence the absorption condition holds.
\end{remark}

\begin{proof}[Proof of Theorem~\ref{thm:dual_strip}]
Under the assumptions of the theorem, both Lemma~\ref{lm:consis} and Proposition~\ref{prop:conti} apply to $\cS$ and $\cS^{\mathrm{b}}$. We first establish the non-crossing property between $\cS$ and $\cS^{\mathrm{b}}$, and then verify~\eqref{eq:def_dual_+}, thereby completing the proof.

\smallskip
\textbf{Non-crossing.} 
Fix $(a_i, r_i) \in \cI$ for $i=1,2$, with $r_1 < r_2$. We show that the trajectories $\cS_{r_1, \cdot}(a_1)$ and $\cS_{r_2, \cdot}^{\mathrm{b}}(a_2)$ almost surely do not cross on $[r_1,r_2]$. 
Equivalently, there do not exist $y, z \in [r_1, r_2]$ with $y < z$ such that
\begin{align}\label{eq:if_cross}
    \left(\cS_{r_1, y}(a_1)-\cS_{r_2, y}^{\mathrm{b}}(a_2)\right)\left(\cS_{r_1, z}(a_1)-\cS_{r_2, z}^{\mathrm{b}}(a_2)\right)<0.
\end{align}
We claim that if~\eqref{eq:if_cross} holds for some $y<z$, then there exist dyadic numbers $y_0<z_0$ in $[y,z]$ such that both processes remain in $(d_1,d_2)$ on $[y_0,z_0]$, and~\eqref{eq:if_cross} still holds with $y,z$ replaced by $y_0,z_0$. We now prove this claim. Let $\varepsilon=\mathrm{sgn}(\cS_{r_1, y}(a_1)-\cS_{r_2, y}^{\mathrm{b}}(a_2))$ with $\mathrm{sgn}(x):=\mathbbm{1}_{\{x>0\}}-\mathbbm{1}_{\{x<0\}}$.
Define $x:=\sup\{y'<z:\, \mathrm{sgn}(\cS_{r_1,y'}(a_1)-\cS_{r_2, y'}^{\mathrm{b}}(a_2))=\varepsilon\}$ and $x':=\inf\{z'>x:\, \mathrm{sgn}(\cS_{r_1,z'}(a_1)-\cS_{r_2, z'}^{\mathrm{b}}(a_2))=-\varepsilon\}$. By continuity of $\cS_{r_1,\cdot}(a_1)$ and $\cS_{r_2,\cdot}^{\mathrm b}(a_2)$ (Proposition~\ref{prop:conti}), $y<x\le x'<z$ and $\cS_{r_1,u}(a_1)=\cS_{r_2, u}^{\mathrm{b}}(a_2)$ for all $u\in [x,x']$. We show that $\cS_{r_1,u}(a_1) \in (d_1,d_2)$ for all $u\in [x,x']$. Suppose this is not the case and that the process touches some boundary $d_i$. Suppose first that $\eta_i=0$. The absorption level $s$ of $\cS_{r_1,\cdot}(a_1)$ satisfies $s\le x'$, and we cannot have $\cS^{\mathrm{b}}_{r_2,s}(a_2)=d_i$ by our assumptions. Hence $s<x$ and $\cS_{r_1,u}(a_1)=d_i$ for all $u\ge s$. In particular, $\cS_{r_1, u}(a_1)-\cS_{r_2, u}^{\mathrm{b}}(a_2)$ does not change sign on $[s,z]$, which contradicts the definition of $x$ and~\eqref{eq:if_cross}. Suppose now that $\eta_i=1$, so that $\eta_i^*=0$. The absorption level $s'$ of $\cS^{\mathrm{b}}_{r_2,\cdot}(a_2)$ is strictly greater than $x'$, and hence $\cS_{r_1, u}(a_1)-\cS_{r_2, u}^{\mathrm{b}}(a_2)$ does not change sign on $[y,s']$, which contradicts the definition of $x'$ and~\eqref{eq:if_cross}. Therefore, $\cS_{r_1,u}(a_1) \in (d_1,d_2)$ for all $u\in [x,x']$. Finally, by continuity of the paths, we can choose $y_0, z_0 \in \bD$ such that $y_0 < x \le x' < z_0$ and both are sufficiently close to $x$ and $x'$, respectively, so as to deduce the claim.

We proceed by contradiction. Suppose that~\eqref{eq:if_cross} holds for some $y < z$ in $[r_1,r_2]\cap \bD$, and that both $\cS_{r_1,\cdot}(a_1)$ and $\cS_{r_2,\cdot}^{\mathrm b}(a_2)$ stay in $(d_1+\frac{1}{N}, d_2-\frac{1}{N})$ on $[y,z]$ for some sufficiently large integer $N \ge 1$. Let $a_1':=\cS_{r_1,y}(a_1)$ and $a_2':=\cS_{r_2,z}^{\mathrm b}(a_2)$. By~\eqref{eq:flow_prop}, almost surely, for any $x\in [y,z]$, $\cS_{r_1,x}(a_1)=\cS_{y,x}(a_1')$ and $\cS_{r_2,x}^{\mathrm b}(a_2)=\cS_{z,x}^{\mathrm b}(a_2')$. By~\eqref{eq:rela}, the coefficient assumptions imply conditions~\ref{item:1c} and~\ref{item:1d} on $I$. By Lemma~\ref{lm:C2}, we may choose $(\sigma_N,b_N)$ on $\r$ satisfying Assumption~\ref{assu:H} and conditions~\ref{item:1a}--\ref{item:1d}, with both $(\sigma_N,b_N)$ and $(\sigma_N^*,b_N^*)$ agreeing with $(\sigma,b)$ and $(\sigma^*,b^*)$, respectively, on $I_N$. Here $\sigma_N^*:=-\sigma_N$ and $b_N^*:=-b_N+\tfrac12\partial_x\bar\sigma_N$. The associated flows do not cross by Proposition~\ref{prop:non-crossing}. Since $\cS_{y,\cdot}(a_1')$ and $\cS_{z,\cdot}^{\mathrm b}(a_2')$ remain in $I_N$ throughout $[y,z]$, they coincide with the corresponding flow lines of~\eqref{eq:SDE_gener} and~\eqref{eq:back_flow} with $(\sigma_N,b_N)$ in place of $(\sigma,b)$. Thus,~\eqref{eq:if_cross} contradicts the non-crossing property. Consequently, for $\overline{\cS}$ defined in~\eqref{def:overline_S}, almost surely the following holds: for any $(a_i,r_i)\in \cI$, $i=1,2$, and any $x\in \r$,
\begin{align}\label{non_cross_for_proof}
	\left[\overline{\cS}_{r_1,x}(a_1)< \overline{\cS}_{r_2,x}(a_2) \right] \,\Longrightarrow \,\left[\forall z\in \r,\ \overline{\cS}_{r_1,z}(a_1)\le \overline{\cS}_{r_2,z}(a_2) \right].
\end{align}

\smallskip
\textbf{Duality relation.} Following the proof of Theorem~\ref{thm:backward_flow_SDE}, it remains to prove that $\cS^{\mathrm{b}}$ coincides with $\cS^*$ defined in~\eqref{eq:def_dual_+}, that is, for any $(b,r)\in I\times\r$ and $x\le r$,
\begin{align}\label{eq:back_SDE_dual_1}
    \cS^{\mathrm b}_{r,x}(b)=\inf\{a\in I:\,\cS_{x,r}(a)>b\}\wedge\sup I.
\end{align}

Suppose $a\in I$ satisfies $\cS_{x,r}(a) \le b$. For any $(b_n,r_n)\in \cI$ with $r_n > r$ and $\cS^{\mathrm{b}}_{r_n,r}(b_n) > b $,~\eqref{eq:cS_I} gives $(a_n,x_n)\in \cI$ with $x_n<x$, $\cS_{x_n,x}(a_n)>a$ and $\cS_{x_n,r}(a_n)\in [\cS_{x,r}(a), \cS^{\mathrm{b}}_{r_n,r}(b_n))$. \eqref{non_cross_for_proof} then implies $\cS^{\mathrm{b}}_{r_n,x}(b_n) \ge \cS_{x_n,x}(a_n)>a$, and applying~\eqref{eq:cS_I} to $\cS^{\mathrm b}$ yields $\cS^{\mathrm{b}}_{r,x}(b) \ge a$.
	
Conversely, suppose $a\in I$ satisfies $\cS_{x,r}(a) > b$. By Lemma~\ref{lm:consis}(i) applied to $\cS^{\mathrm{b}}$, there exist $(b_n,r_n)\in \cI$ with $r_n > r$ such that $\cS^{\mathrm{b}}_{r_n,r}(b_n)\in (b,\cS_{x,r}(a))$. For any $(a_n,x_n)\in \cI$ with $x_n<x$ and $\cS_{x_n,x}(a_n)>a$,~\eqref{eq:cS_I} gives $\cS_{x_n,r}(a_n)\ge \cS_{x,r}(a)>\cS^{\mathrm{b}}_{r_n,r}(b_n)$. Applying~\eqref{non_cross_for_proof}, we obtain $\cS_{x_n,x}(a_n)\ge \cS^{\mathrm{b}}_{r_n,x}(b_n)$, and taking the infimum yields $a\ge \cS^{\mathrm{b}}_{r,x}(b)$. This completes the proof of~\eqref{eq:back_SDE_dual_1}. 
\end{proof}

\section{Applications}\label{sec:apply}
In this section, we illustrate applications of the results in the foregoing sections through three stochastic flow models: the BESQ and Jacobi flows, and a newly introduced PSR flow. It is worth noting that the PSR flow does not fit exactly into the framework of Definition~\ref{def:strip_flow}. Instead, its structure is more closely related to the infinite family of independent coalescing reflected/absorbed Brownian motions (FICRAB) introduced by T\'oth and Werner~\cite{Toth98}, which is crucial in constructing scaling limits of rescaled true self-avoiding walks; see~\cite{Elena25} for recent developments.

\subsection{BESQ flow and Jacobi flow}

Let $I=\r_+$, $b(x,t)=\delta$, and $\sigma(x,t,s)=2\,\mathbbm{1}_{\{0\le s\le x\}}$ in~\eqref{eq:SDE_gener}. The associated general BESQ$^\delta$ flow $\cS$ is defined as the stochastic flow on $\h$ in the sense of Definition~\ref{def:strip_flow}, with boundary condition $\eta_1=0$ if $\delta\le0$, $\eta_1=1$ if $\delta\ge2$, and $\eta_1\in\{0,1\}$ if $\delta\in(0,2)$. This agrees with~\cite[Definitions~3.3~\&~3.4]{elie23}, where condition~\ref{item:R3} is equivalent to property~(iii) of~\cite[Definitions~3.3~\&~3.4]{elie23} by the perfect flow property established in~\cite[Appendix~A]{elie23}. Its dual flow $\cS^*$ is given, for $b\ge 0$ and $x\le r$, by
\begin{align*}
    \cS_{r, x}^*(b):=\inf\{a \ge 0:\, \cS_{x,r}(a)>b \}.
\end{align*}

Let $I=(0,1)$, $b(x,t)=\delta(1-x)-\delta'x$, and $\sigma(x,t,s)=2\,\mathbbm{1}_{\{0\le s\le 1\}}(\mathbbm{1}_{\{s\le x\}}-x)$ in~\eqref{eq:SDE_gener}. The associated general Jacobi$(\delta,\delta')$ flow $\cY$ is defined as the stochastic flow on $I\times\r$ in the sense of Definition~\ref{def:strip_flow}, with boundary conditions
\begin{itemize}[nosep]
    \item $\eta_1=0$ if $\delta\le0$, $\eta_1=1$ if $\delta\ge2$, and $\eta_1\in\{0,1\}$ if $\delta\in(0,2)$;
    \item $\eta_2=0$ if $\delta'\le0$, $\eta_2=1$ if $\delta'\ge2$, and $\eta_2\in\{0,1\}$ if $\delta'\in(0,2)$.
\end{itemize}
When $\delta,\delta'\in(0,2)$, there are four possible choices of boundary conditions. Its dual flow $\cY^*$ is given, for $b\in(0,1)$ and $x\le r$, by
\begin{align*}
    \cY_{r, x}^*(b):=\inf\{a \in (0,1):\, \cY_{x,r}(a)>b \}\wedge 1.
\end{align*}

The existence of the general BESQ flow and the general Jacobi flow follows from~\cite{elie23}. For the general BESQ flow, the coefficients satisfying~\eqref{eq:rela} are given by $b^*(x,t)=2-\delta$ and $\sigma^*(x,t,s)=-2\,\mathbbm{1}_{\{0\le s\le x\}}$. Since squared Bessel processes with positive dimension admit no atoms at deterministic times, the absorption condition follows immediately from Remark~\ref{rmk:sufficient_condi}. For the general Jacobi flow, the absorption condition follows from the transformation relating Jacobi and BESQ flows established in~\cite[Theorem~5.9]{elie23}. Applying Theorem~\ref{thm:dual_strip} therefore yields an explicit description of the dual flows in both cases. For brevity, we say that the flow is killed at $d_i$ if $\eta_i=0$, and non-killed at $d_i$ if $\eta_i=1$.

\begin{theorem}\label{thm:back_BESQ}
Let $\cW^*$ be the image of $\cW$ under $(a,r)\mapsto(a,-r)$.
\begin{enumerate}[label=(\roman*)]
    \item For each $(a,r)\in \h$, $(\cS^*_{-r,-x}(a),\, x\ge r)$ is the almost sure unique strong solution to
    \begin{align*}
        \cS^*_{-r,-x}(a)=a-2\int_r^x \cW^*([0,\cS^*_{-r,-u}(a)], \dd u) + (2-\delta)(x-r), \quad x\ge r.
    \end{align*}
    Let $\mathscr{S}^*_{r,x}(a):=\cS^*_{-r,-x}(a)$ for $x\ge r$. Then ${\mathscr S}^*$ is a general ${\rm BESQ}^{2-\delta}$ flow driven by $-\cW^*$. Moreover, ${\mathscr S}^*$ is killed if $\cS$ is not killed, and it is not killed if $\cS$ is killed. 

    \item For each $(a,r)\in (0,1)\times \r$, $(\cY^*_{-r,-x}(a),\, x\ge r)$ is the almost sure unique strong solution to
    \begin{align*}
        \cY^*_{-r,-x}(a)= a &- 2 \int_r^x\!\! \int_0^1 \bigl(\mathbbm{1}_{\{s \le \cY^*_{-r,-u}(a)\}} - \cY^*_{-r,-u}(a)\bigr)\cW^*(\dd s, \dd u) \\
        &+ \int_r^x \Bigl[(2-\delta)\bigl(1-\cY^*_{-r,-u}(a)\bigr) - (2-\delta')\,\cY^*_{-r,-u}(a)\Bigr]\dd u, \quad x\ge r.
    \end{align*}
    Let $\mathscr{Y}^*_{r,x}(a):=\cY^*_{-r,-x}(a)$ for $x\ge r$. Then ${\mathscr Y}^*$ is a general ${\rm Jacobi}(2-\delta,2-\delta')$ flow driven by $-\cW^*$. Moreover, $\cY^*$ is killed at one of the boundaries $0$ and $1$ if and only if $\cY$ is not killed there. 
\end{enumerate}
\end{theorem}

\subsection{PSR flow and self-duality}

Let $\cW$ be a standard space-time white noise on $\h$ with respect to the natural right-continuous filtration $\cF=(\cF_x,\ x\in\r)$. The following theorem establishes the existence and uniqueness of the polynomially self-repelling (PSR) flow with index $\alpha>0$.

\begin{theorem}\label{thm:forward}
    Fix $\alpha>0$. There exists a family of processes
    $$
    \cS = \left\{\bigl(\cS_{r,x}(a),\, x \ge r\bigr)\right\}_{(a,r) \in \h}
    $$
    such that, almost surely, the regularity properties~\ref{item:R1}--\ref{item:R3} in Definition~\ref{def:flow} (restricted to $\h$) hold, and for each fixed $(a,r)\in \h$, $(\cS_{r,x}(a),\, x\ge r)$ is the unique strong solution to
    \begin{align}\label{eq:SDE_PSR}
        \cS_{r,x}(a) = a + \int_r^x \frac{2}{\cS_{r,u}(a)^{\alpha}}  \int_{\r} s^\alpha \mathbbm{1}_{\{0<s\le \cS_{r,u}(a)\}} \mathcal{W}(\dd s, \dd u) + \frac{x-r}{2\alpha+1},
    \end{align}
    which is absorbed at $\tau_{(a,r)}:=\inf \{x\ge r\vee 0:\, \cS_{r,x}(a)=0\}$. Moreover, this family is unique and continuous, and is referred to as the \emph{polynomially self-repelling flow with index $\alpha$}.
\end{theorem}

Note that~\eqref{eq:SDE_PSR} is an instance of~\eqref{eq:SDE_gener} with
\begin{align}\label{eq:b_and_sig}
    b(x,t)\equiv b(x)=\tfrac{1}{2\alpha+1},\qquad \sigma(x,t,s)\equiv\sigma(x,s)=
    \begin{cases}
        2(s/x)^\alpha, & 0<s\le x,\\
        0, & \text{otherwise}.
    \end{cases}
\end{align}
Then $(\sigma,b)$ satisfy Assumption~\ref{assu:H} and conditions~\ref{item:1a}, \ref{item:1b}, \ref{item:2c}, and~\ref{item:2d} on $\r$, so~\eqref{eq:SDE_PSR} admits a unique strong solution with pathwise uniqueness. 

Let $c=(2\alpha+1)^{-1}$. Each PSR flow line evolves as a $c\,\mathrm{BESQ}^{1}$ process before absorption. The construction is consistent with the estimates in~\cite[Proposition~3.2]{Elena23}. Indeed, consider two flow lines starting at $(M,0)$ and $(M+1,0)$. Write $S_u:=\cS_{0,u}(M)$ and $\cS^+_u:=\cS_{0,u}(M+1)-S_u$. Before the lower flow line is absorbed, their difference has zero drift and quadratic variation
\begin{align*}
    \dd\langle \cS^+\rangle_u
    =4c\left[\cS^+_u+2S_u\left\{1-\left(1+\frac{\cS^+_u}{S_u}\right)^{-\alpha}\right\}\right]\dd u.
\end{align*}
For any fixed $u>0$, as $M\to\infty$, one expects $\cS^+_u/S_u\to0$. Since
$$
\lim_{z\to0}\frac{1-(1+z)^{-\alpha}}{z}=\alpha,
$$
the quadratic variation rate is then asymptotic to $4c(1+2\alpha)\cS^+_u=4\cS^+_u$. This explains why individual flow lines have the $c\,\mathrm{BESQ}^{1}$ law before absorption, while the differences between nearby flow lines at large heights exhibit $\mathrm{BESQ}^{0}$ behavior. Under the half local time convention of \cite{Elena23}, these two scalings align with Remark~3.3 therein.

\begin{remark}\label{rmk:direct_proof}
    While in this paper we adopt a method parallel to that of Section~\ref{subsec:exist} to keep the presentation self-contained, the existence of the PSR flow on $\h^*:=(0,\infty)\times \r$ can also be obtained in a rather direct way. Let $\cShat$ denote the stochastic flow on $\r^2$ associated with $(\sigma,b)$, whose existence is guaranteed by Theorem~\ref{thm:ex}. For each $(a,r)\in \h$, define $\cS_{r,\cdot}(a)$ to be the path $\cShat_{r,\cdot}(a)$ absorbed at $\hat\tau_{(a,r)}:=\inf \{x\ge r\vee 0:\, \cShat_{r,x}(a)=0\}$. Then, $\cS_{r,\cdot}(a)$ is the almost sure unique strong solution to~\eqref{eq:SDE_PSR}, absorbed at $\tau_{(a,r)}$. Moreover, properties~\ref{item:R1} and~\ref{item:R3} are immediate from the definition of $\cS$. By property~\ref{item:R2} of $\cShat$ and the monotonicity of $a\mapsto \cS_{r,x}(a)$, it remains to show that $a\mapsto \cS_{r,x}(a)$ is right-continuous for any $x\ge \tau_{(a,r)}$. To this end, we rely on the following coming-down-from-infinity property, to be established in a forthcoming work: almost surely, for any $x<y$, the set
    $$
    M(x,y):=\{\cShat_{r,y}(a):\, (a,r)\in \cD,\, r<x\}
    $$
    is locally finite. Thus, Corollary~\ref{cor:ex} implies that $\{\cShat_{r,\tau_{(a,r)}}(b):\,b\in\r_+\}$ is locally finite. Since $\cS_{r,\tau_{(a,r)}}(a_n)\downarrow \cS_{r,\tau_{(a,r)}}(a)=0$ as $a_n\downarrow a$, and only finitely many values of $\cS_{r,\tau_{(a,r)}}(a_n)$ lie in $[0,1)$, there exists $N$ such that $\cS_{r,\tau_{(a,r)}}(a_n)=0$ for all $n\ge N$. Hence, $\cS_{r,\cdot}(a_n)$ is absorbed at $0$ at time $\tau_{(a,r)}$ for all $n\ge N$. This outlines a direct pathway to prove Theorem~\ref{thm:forward} on $\h^*$.
\end{remark}

For each $(a,r)\in \cD:=\h\cap \bD^2$, let $\cS_{r,\cdot}(a)$ denote the unique strong solution to~\eqref{eq:SDE_PSR}, absorbed at $\tau_{(a,r)}$. Each $\cS_{r,\cdot}(a)$ is almost surely continuous. Before proving Theorem~\ref{thm:forward}, we study properties of the countable family $\{(\cS_{r,x}(a),\,x\ge r)\}_{(a,r)\in \cD}$.

\begin{lemma}\label{lm:non_cross}
    Almost surely, for any $(a_1,r_1)$, $(a_2,r_2) \in \cD$ and any $z \ge y \ge r_1 \vee r_2$,
    $$
    \cS_{r_1,y}(a_1) < \cS_{r_2,y}(a_2) 
    \ \Longrightarrow\ 
    \cS_{r_1,z}(a_1) \le \cS_{r_2,z}(a_2).
    $$
\end{lemma}

\begin{proof}
    Let $\cShat$ denote the stochastic flow on $\r^2$ constructed in Remark~\ref{rmk:direct_proof}. We claim that for any fixed $(a_1,r_1), (a_2,r_2)\in\cD$, the paths $\cS_{r_1,\cdot}(a_1)$ and $\cS_{r_2,\cdot}(a_2)$ do not cross. That is, there do not exist $z>y\ge r_1\vee r_2$ such that
    \begin{align}\label{eq:cros_if}
        \bigl(\cS_{r_1,y}(a_1)-\cS_{r_2,y}(a_2)\bigr)
        \bigl(\cS_{r_1,z}(a_1)-\cS_{r_2,z}(a_2)\bigr)<0.
    \end{align}
    If $z\le \tau_{(a_i,r_i)}$ for $i=1,2$, then~\eqref{eq:cros_if} is excluded by property~\ref{item:R3} of $\cShat$. By symmetry, suppose~\eqref{eq:cros_if} holds with $z>\tau_{(a_1,r_1)}$. Then $\cS_{r_2,z}(a_2)>\cS_{r_1,z}(a_1)=0$, which implies $z<\tau_{(a_2,r_2)}$ and $\cS_{r_2,z_0}(a_2)>\cS_{r_1,z_0}(a_1)$ for $z_0:=\tau_{(a_1,r_1)}$. It follows that $\cShat_{r_1,y}(a_1)>\cShat_{r_2,y}(a_2)$ while $\cShat_{r_1,z_0}(a_1)<\cShat_{r_2,z_0}(a_2)$, a contradiction. This concludes the proof of the claim.
\end{proof}

\begin{lemma}\label{lm:consi}
	The following properties hold almost surely:
	\begin{enumerate}[label=\textup{(\roman*)}]
		\item For every $x\in \r$, the set $M_+(x):=\{\cS_{r,x}(a):\, r<x,\,(a,r)\in\cD\}$ is dense in $ \r_+$.
		
		\item Fix $(a,r)\in \h$, and let $\cS_{r,\cdot}(a)$ denote the unique strong solution to~\eqref{eq:SDE_PSR}, absorbed at $\tau_{(a,r)}$. Then, almost surely, for all $x\ge r$,
		\begin{align*}
			\cS_{r,x}(a)
			= \inf \left\{ \cS_{r_n,x}(a_n):\, r_n< r,\ \cS_{r_n,r}(a_n)>a,\ (a_n,r_n)\in \cD \right\}.
		\end{align*}
	\end{enumerate}
\end{lemma}

\begin{proof}
The proof of~(i) is identical to that of Lemma~\ref{lm:consis}(i) and is omitted. Let $\{a_n\}_{n\ge1}\subset\bD_+$ be any sequence with $a_n\downarrow a$, and let $\{\cS_{r,\cdot}(a_n)\}_{n\ge1}$ denote the unique strong solutions to~\eqref{eq:SDE_PSR}, absorbed at $\tau_{(a_n,r)}$. Lemma~\ref{lm:non_cross} implies $\tau_{(a_n,r)}\ge \tau_{(a,r)}$. Repeating the proof of Lemma~\ref{prop:density2} and applying Proposition~\ref{thm:conv_appro_2}, it suffices to show that, almost surely, for all $x> \tau_{(a,r)}$,
\begin{align}\label{eq:L1_conv_kill}
    \lim_{n\to\infty} \cS_{r,x}(a_n)=\cS_{r,x}(a)=0.
\end{align}
The monotonicity of the map $\bD_+\ni a\mapsto \cS_{r,x}(a)$, also implied by Lemma~\ref{lm:non_cross}, reduces~\eqref{eq:L1_conv_kill} to showing that for any $\varepsilon>0$, $\lim_{n\to\infty}\p \left(\cS_{r,x}(a_n)>\varepsilon \right)=0$, for all $\delta>0$ and $x= \tau_{(a,r)}+\delta$.
Note that $\tau:=\tau_{(a,r)}$ is a stopping time with respect to $\cF$ and satisfies $\tau<\infty$ almost surely. Let $\p^b(\cS_{r,x}>\varepsilon)$ denote $\p(\cS_{r,x}(b)>\varepsilon)$. The strong Markov property for It\^o diffusions (see, e.g., \cite[Theorem~7.2.4]{ok2013}) yields
\begin{align*}
    \p\! \left(\cS_{r,x}(a_n)>\varepsilon \right)=\e \bigl[ \p \!\left(\cS_{\tau ,x}(\cS_{r,\tau}(a_n))>\varepsilon \,\middle|\, \cF_{\tau} \right) \bigr]
    =\e \bigl[ \p^{\,\cS_{r,\tau}(a_n)}\! \left(\cS_{0 ,\delta}>\varepsilon  \right) \bigr].
\end{align*}

Hence, by Lemma~\ref{lm:non_cross} and the fact that $\lim_{n\to\infty} \cS_{r,\tau}(a_n)=0$, we obtain
\begin{align*}
    \lim_{n\to\infty}\p \left(\cS_{r,x}(a_n)>\varepsilon \right)\le &\lim_{b_n\downarrow 0}\p\bigg(\cS_{0,\delta}(b_n)>\varepsilon \bigg)\\
    \le &\lim_{b_n\downarrow 0}\p \bigg( \inf_{0\le t\le \delta}(2\alpha+1)^{-1}\left[\sqrt{(2\alpha+1)b_n}+B_{t}\right]^2>0\bigg)\\
    = &\lim_{b_n\downarrow 0}\p \bigg(\inf_{0\le t\le \delta}B_t>-\sqrt{(2\alpha+1)b_n} \bigg)=0,
\end{align*}
where $B$ is a standard Brownian motion. The second inequality follows since, on the event $\{\cS_{0,x}(b_n)>\varepsilon\}$, the process $(\cS_{0,t}(b_n),\, 0\le t\le x)$ evolves as a scaled $\mathrm{BESQ}^1$ process, and hence admits the representation $(2\alpha+1)^{-1}\bigl[\sqrt{(2\alpha+1)b_n}+B_{t}\bigr]^2$.
\end{proof}

\begin{proof}[Proof of Theorem~\ref{thm:forward}]
By the proof of Theorem~\ref{thm:ex}, the existence and uniqueness of the PSR flow are guaranteed by Lemmas~\ref{lm:non_cross} and~\ref{lm:consi}, while its continuity follows from the proof of Proposition~\ref{prop:conti}.
\end{proof}

Let $\cS$ be a PSR flow. The dual flow $\cS^*$ is given, for $b\ge0$ and $x\le r$, by
\begin{align}\label{eq:def_dual_+_PSR}
	\cS_{r,x}^*(b):=\inf\{a\ge0:\,\cS_{x,r}(a)>b\}.
\end{align}
The family $\cS^*:=\{(\cS^*_{r,x}(a),\,x\le r)\}_{(a,r)\in\h}$ is called the \emph{dual} (or \emph{backward}) \emph{PSR flow}. In the following, we establish the self-duality of the PSR flow.

\begin{theorem}[Self-duality]
    Let $\cS$ be the PSR flow from Theorem~\ref{thm:forward}, and let $\cS^*$ be its dual defined by~\eqref{eq:def_dual_+_PSR}. Then the families $\{(\cS_{r,x}(a),\,x\ge r)\}_{(a,r)\in\h}$ and $\{(\cS^*_{-r,-x}(a),\,x\ge r)\}_{(a,r)\in\h}$ are identical in law.
\end{theorem}

\begin{proof}
    Recall the pair $(\sigma,b)$ in~\eqref{eq:b_and_sig}. For $x\in\r_+$, we have $b(x)=-b(x)+\tfrac{1}{2}\partial_x \bar\sigma(x)$, and hence $(\sigma^*,b^*):=(-\sigma,b)$ fulfills~\eqref{eq:rela}. Let $(\cS^{\mathrm b}_{-r,-x}(a),\,x\ge r)_{(a,r)\in\h}$ be the PSR flow driven by $-\cW^*$. It suffices to show that $\cS^{\mathrm b}=\cS^*$ almost surely. Since each PSR flow line has the law of a scaled BESQ$^1$ process killed at $\tau_{(a,r)}$, the absorption condition follows from Remark~\ref{rmk:sufficient_condi}. The result then follows by the argument in Theorem~\ref{thm:dual_strip}.
\end{proof}

We call the PSR flow and its dual $(\cS,\cS^*)$ the \emph{double polynomially self-repelling flow} (or \emph{double PSR flow}), and glue them in the sense of~\eqref{def:overline_S} to obtain the two-sided PSR flow $\overline{\cS}$. By the argument in the proof of Theorem~\ref{thm:dual_strip},~\eqref{non_cross_for_proof} holds for all $(a_i,r_i)\in \cD$, $i=1,2$, and all $x\in \r$. Together with the argument in Section~\ref{sec:non-cros}, this yields the following non-crossing property for $\overline{\cS}$.

\begin{proposition}
    Let $\overline{\cS}$ be the two-sided PSR flow. Then, almost surely, for all $(a_i,r_i)\in \h$, $i=1,2$, and all $y,z\in \r$,
    \begin{align*}
        \overline\cS_{r_1,y}(a_1)< \overline\cS_{r_2,y}(a_2) \ \Longrightarrow \ \overline\cS_{r_1,z}(a_1)\le \overline\cS_{r_2,z}(a_2).
    \end{align*}
\end{proposition}

\appendix

\section{Auxiliary convergence results}\label{app:conv}
This section provides supplementary proofs for Section~\ref{subsec:conv_re}. In particular, we establish the following proposition.

\begin{proposition}\label{lm:conv_sigma}
Fix $(a,r)\in\r^2$ and assume that the coefficients $(\sigma_n,b_n)_{n \ge 1}$ and $(\sigma,b)$ satisfy the conditions in Proposition~\ref{thm:conv_appro_1}. Suppose further that, almost surely, $\cW^n \to \hat{\cW}$ in $H^{-d}_{\mathrm{loc}}(\r^2)$ and $X^n \to X$ in $C([r,\infty))$, where $X^n$ is the unique strong solution to
\begin{align}\label{sde_solu_hatS}
    X^n(x) = a +\int_r^x\!\! \int_{\r} \sigma_n(X^n(u), u, s) \cW^n(\dd s, \dd u)+ \int_r^x b_n(X^n(u),u)  \dd u,  \quad x\ge r.
\end{align}
Then, almost surely, the limit $X$ satisfies
\begin{align*}
    X(x) = a + \int_r^x\!\! \int_{\r} \sigma(X(u), u, s)\hat{\cW}(\dd s, \dd u) + \int_r^x b(X(u),u) \dd u,  \quad x\ge r.
\end{align*}
\end{proposition}

The proof of Proposition~\ref{lm:conv_sigma} relies on the following sequence of lemmas. Throughout the rest of this section, we work under the assumptions and notation of Proposition~\ref{lm:conv_sigma}.

\begin{lemma}\label{lem:range}
    For any $T \ge r$, there exists a random compact set $K \subset \r$ such that, almost surely, the ranges of $(X^n)_{n\ge1}$ and $X$ on $[r,T]$ are contained in $K$.
\end{lemma}

\begin{proof}
    Almost surely, $X^n\to X$ uniformly on $[r,T]$. Hence
    $$
    R:=\sup_{n\ge1}\sup_{r\le t\le T}|X^n(t)| \vee \sup_{r\le t\le T}|X(t)|<\infty.
    $$
    Taking $K=[-R,R]$ completes the proof.
\end{proof}

\begin{lemma}\label{lem:L2_convergence_drift}
	For any $T\ge r$, as $n\to\infty$, the following convergence holds almost surely,
	\begin{align*}
		\sup_{r\le x\le T} \left| \left(X^n(x)-\int_r^x b_n(X^n(u),u)\dd u\right) - \left(X(x)-\int_r^x b(X(u),u)\dd u\right) \right| \to 0.
	\end{align*}
\end{lemma}

\begin{proof}
	Fix $T\ge r$. Let $\sup_x$ denote $\sup_{x\in [r,T]}$. Since $X^n\to X$ uniformly on $[r,T]$ almost surely, it suffices to prove $\sup_x |\int_r^x \Delta b_n(u)\dd u|\le \int_r^T |\Delta b_n(u)| \dd u\le (T-r)\sup_x |\Delta b_n(x)|\to 0$ almost surely, where $\Delta b_n(u):=b_n(X^n(u), u) - b(X(u), u)$. By Lemma~\ref{lem:range}, let $\Omega_0$ be an event of probability one such that the ranges of $X^n$ and $X$ on $[r,T]$ are contained in a bounded set $K$ and $X^n(\cdot,\omega) \to X(\cdot,\omega)$ uniformly on $[r,T]$ for all $\omega \in \Omega_0$. For each $\omega\in \Omega_0$, by~\eqref{eq:assume_sig_b} and condition~\ref{item:2d}, there exist $m\ge 1$ such that $K\subset [-m,m]$ and a nondecreasing function $r_m$ on $\r_+$ such that
	\begin{align*}
    	\sup_x |\Delta b_n(x)|
    	&\le \sup_x \Bigl( |b_n(X^n(x),x)-b(X^n(x),x)| 
    	+ |b(X^n(x),x)-b(X(x),x)| \Bigr) \\
    	&\le \sup_x \sup_{z\in K}|b_n(z,x)-b(z,x)| 
    	+ r_m\!\left(\sup_x |X^n(x)-X(x)|\right)
    	\longrightarrow 0,
    \end{align*}
    as $n\to \infty$. This completes the proof.
\end{proof}

For each $n\ge 1$, in view of the SDE~\eqref{sde_solu_hatS}, we define the process $(M^n_x)_{x\ge r}$ by
\begin{align*}
    M_x^n := \int_r^x \!\! \int_{\r} \sigma_n(X^n(u), u, s) \cW^n(\dd s, \dd u)
    = X^n(x) - a - \int_r^x b_n(X^n(u),u) \dd u.
\end{align*}

\noindent
By~\eqref{eq:assume_lin_bound} and~\eqref{eq:lm_tight}, and the Burkholder--Davis--Gundy (BDG) inequality (see, e.g., \cite[Chapter~IV.4]{RevYor}), there exists a constant $C_{\mathrm{B}}>0$ such that
\begin{align}\label{eq:L_4_bound1}
	\nonumber \sup_n \e\bigl[(M_x^n)^4\bigr] 
	\le&\; C_{\mathrm{B}} \sup_n \e\!\left[\Bigl(\int_r^x\!\! \int_{\r} \sigma_n(X^n(u), u, s)^2 \dd s\dd u \Bigr)^2\right]\\
	\le &\; C_{\mathrm{B}} (K_\sigma')^2 (x-r)^2 \sup_n \hat\e\!\left[\Bigl(1+\sup_{r\le u\le x} |X^n(u)|\Bigr)^4\right]<\infty.
\end{align}
Therefore, $(M_x^n)_{x \ge r}$ is a fourth-integrable martingale with respect to the filtration $\cF^n_x:=\sigma\bigl(X^n(y) : r \le y \le x\bigr)$.
Lemma~\ref{lem:L2_convergence_drift} implies that, almost surely, $(M_x^n)_{x \ge r}$ converges in $C([r,\infty))$. Denote the limit by $(M_x)_{x\ge r}$, then, almost surely, for all $x\ge r$,
\begin{align}\label{eq:M_x=fcS}
    M_x = X(x) - a - \int_r^x b(X(u),u) \dd u.
\end{align}
Moreover, the uniform $L^4$-boundedness of $(M_x^n)_{n \ge 1}$ implies that it is uniformly integrable for each fixed $x$. Hence, the limit $M_x$ is integrable. 

\begin{lemma}\label{lm:mart_M}
    The process $(M_x)_{x\ge r}$ is a $\tilde{\cF}$-martingale with quadratic variation
    $$
    \langle M \rangle_x=\int_r^x\!\! \int_{\r} \sigma(X(u), u, s)^2 \dd s\dd u.
    $$
\end{lemma}

\begin{proof}
	Define $\tilde\cF_x := \sigma\bigl(X(y) : r \le y \le x\bigr)$. From the identity~\eqref{eq:M_x=fcS}, $M_x$ is measurable with respect to $\tilde{\cF}$. To show that $M_x$ is a $\tilde\cF$-martingale, it suffices, by the monotone class theorem, to verify that for any integer $k$, any $r \le t_1 < \cdots < t_k \le y < x$, and any bounded continuous function $f:\r^k \to \r$, 
    \begin{align}\label{eq:mar_prop}
        \e \!\left[(M_x - M_y) f\bigl(X_{t_1}, \ldots, X_{t_k}\bigr)\right] = 0.
    \end{align}
    Set $Z := f\bigl(X_{t_1}, \ldots, X_{t_k}\bigr)$ and $Z_n := f\bigl(X^n_{t_1}, \ldots, X^n_{t_k}\bigr)$. Since $(M^n_x)_{x \ge r}$ is a $\cF^n$-martingale and $Z_n$ is bounded, it follows that for all $n \ge 1$,
    \begin{align*}
        \e\!\left[ (M^n_x - M^n_y) Z_n \right] = 0.
    \end{align*}
    By the almost sure uniform convergence of $M^n$ and $X^n$ on compact intervals and the continuity of $f$, we have $(M^n_x - M^n_y)Z_n \to (M_x - M_y)Z$ almost surely.
    Moreover, by~\eqref{eq:L_4_bound1}, $\{(M^n_x - M^n_y) Z_n\}_{n \ge 1}$ is uniformly integrable. Thus, 
    \begin{align*}
        \e \!\left[(M_x - M_y) Z\right]
        = \lim_{n \to \infty} \e\!\left[ (M^n_x - M^n_y) Z_n \right]
        = 0.
    \end{align*}

	\smallskip
    For each $n$, the process $\tilde M^n_x:=(M^n_x)^2-\langle M^n\rangle_x$, $x\ge r$, is a martingale, where
    \begin{align*}
        \langle M^n \rangle_x=\int_r^x\!\! \int_\r \sigma_n(X^n(u),u,s)^2\dd s\dd u.
    \end{align*}
    Fix $T\ge r$. By Lemma~\ref{lem:range}, there exists a bounded random set $K$ such that $X^n([r,T])\subset K$ for all $n\ge1$ and $X([r,T])\subset K$. For any $x \in [r,T]$,
    \begin{align*}
        \left| \langle M^n \rangle_x - \int_r^x\!\! \int_\r \sigma(X(u),u,s)^2 \dd s \dd u \right| \le& \int_r^x \!\! \int_\r \left| \sigma_n(X^n(u),u,s)^2 - \sigma(X(u),u,s)^2 \right| \dd s \dd u \\
        \le&\; (T-r)\sup_{r \le u \le T} \sqrt{\Xi_n^+(u) \Xi_n^-(u)},
    \end{align*}
    where 
    $$
    \Xi_n^+(u):=\int_\r \bigl( \sigma_n(X^n(u),u,s) + \sigma(X(u),u,s) \bigr)^2 \dd s\le 2\bar\sigma_n(X^n(u),u)+2\bar\sigma(X(u),u)
    $$
    is uniformly bounded by condition~\ref{item:2c} and~\eqref{eq:assume_lin_bound}, and
    \begin{multline}
        \Xi_n^-(u):=\int_\r \bigl( \sigma_n(X^n(u),u,s) - \sigma(X(u),u,s) \bigr)^2 \dd s\\
        \le 2 \sup_{z\in K}\int_\r \bigl( \sigma_n(z,u,s) - \sigma(z,u,s) \bigr)^2 \dd s+ 2\int_\r \bigl( \sigma(X^n(u),u,s) - \sigma(X(u),u,s) \bigr)^2 \dd s. \label{es_eq2}
    \end{multline}
    By~\eqref{eq:assume_sig_b}, the first term vanishes uniformly over $u\in [r,T]$, while condition~\ref{item:1b} implies that there exist $m\ge1$ with $K\subset[-m,m]$ and a nondecreasing function $\rho_m$ on $\r_+$ such that the second term is bounded by $\rho_m\!\left(\sup_{r\le u\le T}|X^n(u)-X(u)|\right)^2$, which converges to zero. Therefore, $\langle M^n \rangle_x$ converges almost surely in $C([r,\infty))$ to
	$$A_x := \int_r^x \!\! \int_\r \sigma(X(u),u,s)^2\dd s \dd u.$$

    \smallskip
	By~\eqref{eq:L_4_bound1}, the BDG inequality implies that
	\begin{align*}
		\sup_n \e\!\left[\bigl(\tilde M^n_x\bigr)^2\right]
		\le 2\sup_n \e[(M_x^n)^4] + 2\sup_n \e[\langle M^n\rangle_x^2] \le 2(C_{\mathrm{B}}+1) \sup_n \e[\langle M^n\rangle_x^2] < \infty.
	\end{align*}
    It follows that for any fixed $x$, the family $\bigl(\tilde M^n_x\bigr)_{n\ge 1}$ is uniformly integrable. Combining this with the almost sure uniform convergence of $M^n$ to $M$ and $\langle M^n \rangle$ to $A$ on compact intervals, we conclude that $\tilde M^n_x$ converges in $C([r,\infty))$ to a process $\tilde M$, which preserves the martingale property~\eqref{eq:mar_prop} (with $M$ replaced by $\tilde M$) and satisfies
	\begin{align*}
		-(M_x)^2=-\tilde M_x-A_x.
	\end{align*}
    The left-hand side is a $\tilde\cF$-supermartingale with $M_r=0$. Moreover, $A$ is a continuous, nondecreasing, $\tilde\cF$-predictable process with $A_r = 0$. By the Doob--Meyer decomposition (see, e.g., \cite[Chapter~III, Theorem~16]{prot05}), the decomposition is unique. Therefore, $A$ coincides almost surely with the quadratic variation of $M$. This completes the proof.
\end{proof}

For any function $\phi \in C_c^\infty(\r^2)$ and any $x\ge r$, define the martingales
\begin{align*}
	N_x^n(\phi) := \int_r^x \!\! \int_{\r} \phi(u, s) \cW^n(\dd s, \dd u), \qquad 
	N_x(\phi) :=\int_r^x \!\! \int_{\r} \phi(u, s)\hat{\cW}(\dd s, \dd u).
\end{align*}
The almost sure convergence $\cW^n \to \hat{\cW}$ in $H^{-d}_{\mathrm{loc}}(\r^2)$ implies that, for every $f \in C_c^\infty(\r^2)$
\begin{equation}\label{conv_S}
    \langle \cW^n, f \rangle := \iint_{\r^2} f(u,s) \cW^n(\dd s, \dd u) \longrightarrow \langle \hat{\cW}, f \rangle := \iint_{\r^2} f(u,s) \hat{\cW}(\dd s, \dd u),
\end{equation}
almost surely. Although $\phi(u,s)\mathbbm{1}_{[r,x]}(u)$ does not belong to $C_c^\infty(\r^2)$, the indicator $\mathbbm{1}_{[r,x]}$ can be approximated by smooth compactly supported functions. Combined with the It\^o isometry for the white noises $\cW^n$ and $\hat{\cW}$, this extends \eqref{conv_S} to integrands of the form $\phi(u,s)\mathbbm{1}_{[r,x]}(u)$. Consequently, for each fixed $x\ge r$, $N^n_x(\phi)\to N_x(\phi)$ in probability.

We define the joint filtration $\cG$ by
$$
\cG_x:=\tilde{\cF}_x \vee \sigma\bigl(\hat\cW(B): B\in \mathcal{B}([r,x]\times \r)\bigr), \qquad x\ge r.
$$
To see that the martingale properties are preserved under this enlargement, fix $r \le u < x$ and let $\Phi$ be any bounded continuous functional depending on finitely many variables of the form $X^n(t_i)$ and $\langle \cW^n, \phi_j \rangle$, where $t_i \le u$ and $\phi_j$ are smooth with compact support in $[r,u]\times \r$. Since $M^n$ is a martingale with respect to the natural filtration of $\cW^n$, 
$$
\hat\e\!\left[ (M_x^n - M_u^n)\Phi \right] = 0.
$$
By the almost sure convergence of $X^n$ and $\cW^n$ (in their respective topologies) alongside the uniform integrability of $M^n$, passing to the limit yields that $M$ is a continuous martingale with respect to $\cG$. The same argument verifies that $\hat\cW$ remains a martingale measure with respect to $\cG$.

\begin{lemma}
	The quadratic covariation of $M$ and $N(\phi)$ satisfies
	\begin{align}\label{eq:step_1}
		\langle M, N(\phi) \rangle_x = \int_r^x \!\! \int_{\r} \sigma(X(u), u, s)\, \phi(u, s) \, \dd s \dd u.
	\end{align}
\end{lemma}

\begin{proof}
	For each $n$,
	$$
	\langle M^n, N^n(\phi) \rangle_x = \int_r^x\!\! \int_\r \sigma_n(X^n(u), u, s) \phi(u,s)\dd s \dd u
	$$
	and It\^{o}'s product rule gives 
	$$
	M^n_x N^n_x(\phi) = M^n_r N^n_r(\phi) + \int_r^x M^n_u \dd N^n_u(\phi) + \int_r^x N^n_u (\phi)\dd M^n_u + \langle M^n, N^n(\phi) \rangle_x.
	$$
    Hence
    \begin{align*}
	Y_x^n(\phi):=M_x^n N_x^n(\phi) - \langle M^n, N^n(\phi) \rangle_x
	\end{align*}
    is a martingale. Let $C_\phi := \|\phi\|_\infty^2 \cdot \operatorname{meas}(\operatorname{supp}_s(\phi))$.
	For any $T\ge r$, the Cauchy--Schwarz inequality yields
	\begin{align*}
		&\left| \langle M^n, N^n(\phi) \rangle_x - \int_r^x \!\! \int_{\r}  \sigma(X(u), u, s)\,\phi(u, s)\, \dd s \dd u \right|^2 \\
		&\le  (T-r)^2 C_\phi \sup_{r\le u\le T} \int_\r \bigl(\sigma_n(X^n(u), u, s)-\sigma(X(u), u, s)\bigr)^2 \dd s
	\end{align*}
	converges to $0$ uniformly on $[r,T]$ by~\eqref{es_eq2} and the subsequent argument. Let
    \begin{align*}
		Y_x(\phi):=M_x N_x(\phi) -  \int_r^x \!\! \int_{\r}  \sigma(X(u), u, s) \phi(u, s) \dd s \dd u.
	\end{align*}
	Then, for each fixed $x\ge r$, $Y_x^n(\phi)$ converges in probability to $Y_x(\phi)$.
	
	We establish uniform integrability of the family $\bigl(Y_x^n(\phi)\bigr)_{n\ge 1}$ by the following estimate, together with the previous arguments: for each $n\ge 1$,
	\begin{align*}
		\e \bigl[(Y_x^n(\phi))^2\bigr] 
		\le \e \bigl[(M_x^n)^4\bigr] + \e \bigl[(N_x^n(\phi))^4\bigr] + 2 (T-r)^2 C_\phi \e\! \left[\sup_{r\le u\le T} \bar\sigma_n(X^n(u), u) \right].
	\end{align*}
	Arguing as in the proof of Lemma~\ref{lm:mart_M}, we conclude that $(Y_x(\phi))_{x\ge r}$ is a $\cG$-martingale. Since $M$ and $N(\phi)$ are continuous $\cG$-martingales, the process $M_xN_x(\phi)-\langle M,N(\phi)\rangle_x$ is also a continuous $\cG$-martingale. Therefore, 
    $$
    D_x:=\langle M, N(\phi)\rangle_x-\int_r^x \int_{\r} \sigma(X(u), u, s) \phi(u, s) \dd s \dd u
    $$
    is a continuous $\cG$-martingale with $D_r=0$. Since both terms in $D$ have finite variation on compact intervals, $D \equiv 0$. This proves~\eqref{eq:step_1}. 
\end{proof}

The following lemma establishes the well-posedness of the stochastic integral 
$$\int_r^x\! \int_{\r} \sigma(X(u), u, s)\hat{\cW}(\dd s, \dd u).$$
Without loss of generality, take $[r,x]=[0,T]$. Based on \cite[pp.~292--295]{Walsh1986AnIT}, let $\mathcal{P}$ denote the Hilbert space of $\cG$-predictable processes $f$ for which $\hat\e\!\left[\int_0^T\!\! \int_{\r} f(u, s)^2 \dd s \dd u\right]<\infty$.

\begin{lemma}\label{lm:predic}
    The process $H(u,s) := \sigma(X(u), u, s)$ belongs to $\mathcal{P}$.
\end{lemma}

\begin{proof}
    Let $K$ be as in Lemma~\ref{lem:range}, and set $R_T:=\sup_{0\le u\le T}|X(u)|$. By~\eqref{eq:lm_tight} and Fatou's lemma, $\e[R_T^4]\le \liminf_{n\to\infty}\e\!\left[\sup_{0\le u\le T}|X^n(u)|^4\right]<\infty$. Condition~\ref{item:2c} therefore gives
    $$
    \e\!\left[\int_0^T\!\!\int_\r H(u,s)^2 \dd s\dd u\right]
    \le CT\,\e[(1+R_T)^2]<\infty.
    $$

    For each $n \in \N$, let $u^n_k := k 2^{-n}$ for $k \ge 0$, and define
    \begin{align*}
        H_n(u,s) = \sum_{k\ge 0} \mathbbm{1}_{(u^n_{k},\, u^n_{k+1}]}(u) H(u^n_{k},s).
    \end{align*}
    Since $X$ is $\cG$-adapted, the mapping $(\omega,s)\mapsto H(u^n_k,s)$ is $\cG_{u^n_k}\otimes \mathcal B(\r)$-measurable for each $k$. It then follows from the definition of predictable processes that $H_n(u,s)$ is $\cG$-predictable for every $n \in \N$. It remains to prove that
    \begin{align*}
        \lim_{n\to \infty}\hat\e\!\left[\int_0^T\!\! \int_{\r} \bigl(H_n(u,s)-H(u,s)\bigr)^2 \dd s \dd u\right]=0.
    \end{align*}
    For almost every $\omega$ and any $u,v \in [0,T]$, since $X$ is uniformly continuous on $[0,T]$ with values in the compact set $K(\omega)$, condition~\ref{item:1b} and~\eqref{assu:conti_t} imply $\int_{\r} \bigl(H(u,s)-H(v,s)\bigr)^2 \dd s\to 0$ uniformly as $|u-v|\to 0$. Thus, almost surely,
    \begin{align*}
        \int_0^T\!\! \int_{\r} \bigl(H_n(u,s)-H(u,s)\bigr)^2 \dd s \dd u= \sum_{k\ge 0} \int_{u^n_{k}\wedge T}^{u^n_{k+1}\wedge T} \!\! \dd u \int_{\r} \bigl(H(u^n_k,s)-H(u,s)\bigr)^2 \dd s \to 0.
    \end{align*}
    Finally, by condition~\ref{item:2c}, $\int_0^T\!\!\int_\r |H_n(u,s)-H(u,s)|^2 \dd s \dd u \le 4CT(1+R_T)^2$. Since the right-hand side is integrable, the proof is completed by the dominated convergence theorem.
\end{proof}

\smallskip
In the remainder of this section, we aim to show that the martingale $M$ admits a stochastic integral representation with respect to $\hat\cW$, that is, almost surely,
\begin{align*}
	M_x = \int_r^x\!\! \int_{\r} \sigma(X(u), u, s)\hat{\cW}(\dd s, \dd u)=: I_x,\qquad \forall x \ge r.
\end{align*}
Once this is established, the proof of Proposition~\ref{lm:conv_sigma} is complete.

\begin{lemma}\label{lm:qua_cov}
	The quadratic covariation of $M$ and $I$ satisfies
	\begin{align*}
		\langle M,I \rangle_x=\int_r^x\!\! \int_{\r} \sigma(X(u), u, s)^2 \dd s\dd u.
	\end{align*}
\end{lemma}

Lemma~\ref{lm:qua_cov}, together with Lemma~\ref{lm:mart_M} and the definition of $I_x$, implies
\begin{align*}
	\langle M-I \rangle_x=\langle M \rangle_x + \langle I \rangle_x - 2 \langle M,I \rangle_x= 0.
\end{align*}
Since $M$ and $I$ are continuous martingales vanishing at $r$, $(M_x-I_x)^2-\langle M-I \rangle_x=(M_x-I_x)^2$ is a martingale. In particular, $\e [(M_x-I_x)^2]=\e [(M_r-I_r)^2]=0$, which yields $M_x = I_x$ almost surely for all $x \ge r$, achieving the representation we aimed for. 

\begin{proof}[Proof of Lemma~\ref{lm:qua_cov}]
	Consider the simple predictable process $h$ defined by
	\begin{equation*}
		h(x,s) := \sum_{i=1}^{n} \xi_i \mathbbm{1}_{(t_{i-1}, t_i]}(x) \phi_i(x,s),
	\end{equation*}
	where $r = t_0 < t_1 < \dots < t_n < \infty$, each $\phi_i \in C_c^\infty(\r^2)$ is deterministic, and $\xi_i$ is a bounded $\cG_{t_{i-1}}$-measurable random variable. Define for any $x\ge r$,
	\begin{align*}
		N_x(h):=\int_r^x \!\!\int_\r h(u,s) \hat\cW(\dd s,\dd u)=\sum_{i=1}^n \xi_i \int_{t_{i-1}\wedge x}^{t_i\wedge x}\int_\r \phi_i(u,s) \hat\cW(\dd s, \dd u).
	\end{align*}
	
	\noindent
	By linearity, we extend the identity~\eqref{eq:step_1} to any simple predictable process $h$, and obtain
	\begin{align*}
		\langle M, N(h)\rangle_x 
		&= \sum_{i=1}^n \xi_i \int_{t_{i-1}\wedge x}^{t_i\wedge x}\int_\r \sigma(X(u), u, s) \phi_i(u,s) \dd s\dd u \\
		&= \int_r^x \!\! \int_\r \sigma(X(u), u, s) h(u,s)\dd s\dd u.
	\end{align*}
	
	Since $M$ and $I=N(H)$ are continuous martingales, both $\int_r^x\! \int_\r H(u,s)^2 \dd s \dd u$ and $\langle M, I\rangle_x$ are continuous in $x$. Thus, it suffices to show that, for fixed $x$,
	\begin{align}\label{eq:as_id}
		\langle M, N(H)\rangle_x = \int_r^x\!\! \int_\r H(u,s)^2 \dd s \dd u \quad \text{almost surely.}
	\end{align}
	By Lemma~\ref{lm:predic}, choose simple predictable processes $(h^n)_{n\ge1}$ such that
	$$
	\lim_{n \to \infty} \e \left[\int_r^x\!\! \int_\r \left|H(u, s)-h^n(u, s)\right|^2 \dd s \dd u\right]=0.
	$$
	
	\noindent 
	Therefore, $N_x(h^n) \to N_x(H)$ in $L^2$, which implies $\langle M, N(h^n) \rangle_x \to \langle M, N(H) \rangle_x$ in $L^1$. On the other hand,  $\int_r^x\! \int_\r H(u,s) h^n(u,s) \dd s \dd u\to \int_r^x\! \int_\r H(u,s)^2 \dd s \dd u$ in $L^1$. Hence, there exists a subsequence $\bigl\{h^{n_k}\bigr\}_{k\ge 1}$ along which both sequences converge almost surely, yielding the desired identity~\eqref{eq:as_id}.
\end{proof}

\section{Approximation of the coefficients}\label{app:appro}
In this section we introduce two approximation schemes for the coefficients $(\sigma,b)$. 
Before stating the corresponding lemmas, we recall the standard cutoff functions and mollifiers that will be used in Lemma~\ref{lm:a1} and~\ref{lm:C1}.

Choose a cutoff function $\zeta\in C_c^\infty(\r)$ such that
$$
\zeta(x)= 1 \ \text{ for } |x|\le 1,\quad \operatorname{supp} \zeta= [-2,2],\quad  0\le \zeta\le 1.
$$
Let $D_n := [-n,n]$ and define $\zeta_n(x):=\zeta(x/n)$. 
Then $\zeta_n\equiv 1$ on $D_n$ and $\zeta_n\equiv 0$ on $\r\setminus D_{2n}$. Its derivative satisfies $|\zeta_n'(x)|\le \tfrac{1}{n}\sup_{u\in \r} |\zeta'(u)|=:\tfrac{B}{n}$.
Let $\eta\in C_c^\infty(\r)$ be the standard mollifier supported on $D_1$, and set $\eta_n(x):=n\eta(nx)$. 
By standard results on convolution and smoothing (see, e.g., \cite[Appendix~C.5]{Evans}), if $f\in C(U)$ for an open set $U\subset \r$, then $\eta_n * f \to f$ uniformly on any compact subsets of $U$.

\begin{lemma}\label{lm:a1}
Suppose that $(\sigma,b)$ satisfy Assumption~\ref{assu:H} and conditions~\ref{item:1a}, \ref{item:1b}, \ref{item:2c}, and \ref{item:2d}. Then there exists a sequence $(\sigma_n,b_n)_{n\ge1}$ such that, for any $T\ge r$,
\begin{enumerate}[label=(\roman*)]
    \item \eqref{eq:assume_sig_b} holds for any compact set $K \subset \r$;
    \item \eqref{eq:assume_lin_bound} holds for all $x\in \r$ and $t\in [r,T]$ with constants $K_b'$ and $K_\sigma'$ independent of $n$;
    \item for each $n$, $(\sigma_n,b_n)$ satisfies Assumptions~\ref{assu:H} and~\ref{assu:A} for $t\in [r,T]$.
\end{enumerate}
\end{lemma}

\begin{proof}
    For each $n\ge1$, define
    \begin{align*}
        \sigma_n(x,t,s) := \zeta_n(x) \int_{\r} \eta_n(u) \sigma(x-u,t, s) \dd u, \quad
        b_n(x,t) := \zeta_n(x) \int_{\r} \eta_n(u) b(x-u,t) \dd u.
    \end{align*}
    Fix a compact set $K \subset \r$ and choose $N$ with $K \subset D_{N-1}$. For any $n \ge N$, $\zeta_n(x) \equiv 1$ on $K$. Thus, for any $x \in K$ and $t\in [r,T]$, applying the Cauchy--Schwarz inequality yields
    \begin{align}\label{Sig_22}
    \int_\r (\sigma_n(x,t,s)-\sigma(x,t,s))^2\dd s
    = &
    \int_{\r} \Bigl(
        \int_{\r} \eta_n(u)\bigl(\sigma(x-u,t,s)-\sigma(x,t,s)\bigr)\dd u
    \Bigr)^2 \dd s \notag\\
    \le &
    \int_{\r} \int_{\r} 
    \eta_n(u)\bigl(\sigma(x-u,t,s)-\sigma(x,t,s)\bigr)^2
    \dd s\dd u \notag\\
    \le &
    \int_{\r} \eta_n(u) \rho_{N}(|u|)^2\dd u
    =
    \int_{\r} \eta(v) \rho_{N}(|v|/n)^2\dd v.
    \end{align}
    The last inequality follows from condition~\ref{item:1b}. Since $\rho_N(0+)=0$, the right-hand side converges to zero by dominated convergence. This proves (i) for $\sigma_n$.

    For all $x\in \r$, $t\in [r,T]$ and integers $n$, by condition~\ref{item:2c}, we have
    \begin{multline}
        \int_\r \sigma_n(x,t,s)^2\dd s\le \int_\r \Bigl( \int_\r \eta_n(u)\sigma(x-u,t,s) \dd u \Bigr)^2\dd s
        \le \int_\r \int_\r \eta_n(u)\sigma(x-u,t,s)^2 \dd s \dd u\\
        \le K_\sigma \int_\r \eta_n(u)(1+|x-u|)^2 \dd u 
        \le K_\sigma (2+|x|)^2\le 4K_\sigma (1+|x|)^2.\ \ \label{bound_sig_n}
    \end{multline}
    Therefore, the second bound in~\eqref{eq:assume_lin_bound} holds with $K_\sigma' = 4K_\sigma$, which proves \textnormal{(ii)} for $\sigma_n$.
    
    Fix $n \ge 1$. Then $\sigma_n$ is smooth in $x$. By~\eqref{bound_sig_n} and the Cauchy--Schwarz inequality, the first bound in~\ref{item:A2} holds. To verify that $\int_\r \left(\partial_x^\beta \sigma_n(x,t,s)\right)^2\dd s$ is uniformly bounded for each $\beta = 1,2,3$, we define, for any $\beta\ge 0$,
    \begin{align*}
        F_\beta^n (x,t):=\int_\r \Bigl(\frac{\partial^\beta}{\partial x^\beta} (\eta_n*\sigma)(x,t,s) \Bigr)^2\dd s= \int_\r \Bigl( \int_\r \eta_n^{(\beta)}(u)\sigma(x-u,t,s) \dd u \Bigr)^2\dd s.
    \end{align*}
    Since the derivatives of $\zeta_n$ are bounded, the Leibniz rule reduces the verification to showing that $F_\beta^n$ is uniformly bounded on $D_{2n}\times [r,T]$ for each $\beta$. By condition~\ref{item:2c}, 
    \begin{align}\label{eq:a4}
        \nonumber F_\beta^n (x,t)\le& \int_\r \bigl| \eta_n^{(\beta)}(u)\bigr|^2  \dd u \cdot \int_\r\int_{|u|\le \frac{1}{n}} \sigma(x-u,t,s)^2 \dd u \dd s\\
        = &\; C_\beta n^{2\beta+1}\int_{|u|\le \frac{1}{n}}\bar\sigma(x-u,t) \dd u
        \le 2 C_\beta K_\sigma n^{2\beta}(2+|x|)^2,
    \end{align}
    where $C_\beta:=\int_\r |\eta^{(\beta)}(u)|^2\dd u$ is a constant. Thus, $F_\beta^n (x,t)\le 32 C_\beta K_\sigma n^{2\beta+2}$ for any $x\in D_{2n}$ and $t\in [r,T]$. This establishes~\ref{item:A2}, and hence Assumption~\ref{assu:A} for $\sigma_n$.
    Using~\eqref{assu:conti_t} for $\sigma$, 
    \begin{align}\label{eq:conti_sig_n_t}
    \nonumber&\int_\r \Bigl(\int_{\r} \eta_n(u)\big[\sigma(x-u,t+h,s)-\sigma(x-u,t,s)\big]\dd u\Bigr)^2 \dd s \\
    &\qquad \le C_0 n \int_{|u|\le 1} \int_{\r} \big[\sigma(x-u,t+h,s)-\sigma(x-u,t,s)\big]^2 \dd s \dd u\to 0, \quad \text{as } h \to 0 .
    \end{align}
    This verifies Assumption~\ref{assu:H} for $\sigma_n$, completing the proof for $\sigma_n$.
    
    \smallskip
    The proofs of (i) and (ii) for $b_n$ are analogous and simpler, so we only verify (iii). For each fixed $n$,
    $$|b_n(x,t)|\le K_b\int_{\r} \eta_n(u)(1+|x-u|)\dd u\le K_b\int_{\r} \eta_n(u)(2+|x|)\dd u\le 2K_b(1+|x|),$$
    so $b_n$ satisfies~\ref{item:1a}. Since $b_n$ is smooth in the first variable with compact support, its derivatives $\partial_x b_n$ and $\partial_x^2 b_n$ are uniformly bounded for $t\in [r,T]$, and~\ref{item:A1} follows. To prove the continuity of $b_n$, it remains to show that $|b_n(x,t)-b_n(x,t')|\to 0$ as $|t-t'|\to 0$. Indeed,
    $$|b_n(x,t) - b_n(x,t')| \le \mathbbm{1}_{\{|x|\le 2n\}} \int_{\r} \eta_n(u) |b(x-u,t) - b(x-u,t')| \dd u.$$
    The integrand is bounded by $2K_b\eta_n(u)(1+|x-u|)$, which is integrable. The desired continuity follows from the dominated convergence theorem and the continuity of $b$.
\end{proof}

\begin{lemma}\label{lm:C1}
Suppose that $(\sigma,b)$ satisfy Assumption~\ref{assu:H} and conditions~\ref{item:1a}--\ref{item:1d}. Then there exists a sequence $(\sigma_n,b_n)_{n\ge1}$ such that, for any $T\ge r$,
\begin{enumerate}[label=(\roman*)]
    \item \eqref{eq:assume_sig_b} and~\eqref{eq:assume_sig_partial} hold for any compact set $K \subset \r$;
    \item \eqref{eq:assume_lin_bound_1} holds for all $x\in \r$ and $t\in [r,T]$ with constants $K_b'$ and $K_\sigma'$ independent of $n$;
    \item for each $n$, $(\sigma_n,b_n)$ satisfies Assumptions~\ref{assu:H} and~\ref{assu:A} for $t\in [r,T]$.
\end{enumerate}
\end{lemma}

\begin{proof}
    The drift $b$ can be handled as in Lemma~\ref{lm:a1}, so it suffices to focus on the diffusion coefficient $\sigma$. Define an auxiliary family of functions $\{\sigma_{n,m}^k\}$:
    First, we mollify $\sigma$ in the $x$-variable by setting
    \begin{align*}
    \tilde{\sigma}_m(x,t,s) := \int_\r \eta_m(u)\sigma(x-u,t,s) \dd u.
    \end{align*}
    For $x\in D_n$ and $t\in [r,T]$, define
    \begin{equation}\label{def:sig_k}
        \tilde{\sigma}_{n,m}^k(x,t,s) := 
        \begin{cases} 
            \tilde{\sigma}_m(x,t,s) & \text{if } |s| \le m, \\
            \sqrt{\int_\r \eta_k(u) [C_{n,m}(t)-f_{m}(x-u,t)]\dd u} & \text{if } m+1 \le s \le m+2, \\
            0 & \text{otherwise},
        \end{cases}
    \end{equation}
    and extend it by zero for $x \notin D_n$, where $C_{n,m}(t):=\sup_{x\in D_{n+1}}|f_{m}(x,t)|+\tfrac{1}{m}$, and
    \begin{gather*}
        f_{m}(x,t):=\int_{|s|\le m} \tilde\sigma_{m}(x,t,s)^2\dd s-\bar \sigma(x,t).
    \end{gather*}
    Finally, let
    \begin{align*}
        \sigma_{n,m}^k(x,t,s) := \zeta_{n/2}(x)\,\tilde{\sigma}_{n,m}^k(x,t,s).
    \end{align*}
    
    Fix indices $n,m,k\in \z_+$. Recall that $F_\beta^m(x,t):= \int_{\r} \left(\partial_x^\beta \tilde\sigma_m(x,t,s)\right)^2\dd s$ is bounded on $D_{n+1}\times [r,T]$ for any $\beta \ge 0$ by~\eqref{eq:a4}. Following the proof of Lemma~\ref{lm:es_conA}(i), for each $t$ the map $x \mapsto \int_{|s|\le m} \tilde\sigma_m(x,t,s)^2 \dd s$ is smooth on $D_{n+1}$. By condition~\ref{item:1c}, it follows that $f_m(\cdot,t)\in C^1(\r)$, so $C_{n,m}(t)$ is finite.
    Moreover, since $g_{n,m}(x,t) := C_{n,m}(t) - f_m(x,t)$ satisfies $\frac{1}{m}\le g_{n,m}(x,t)\le 2 C_{n,m}(t)$ on $D_{n+1}$, the convolution term $\Phi_{n,m}^k(x,t):=\eta_k * g_{n,m}(x,t)$ is strictly positive and bounded above by $2C_{n,m}(t)$ on $D_n$. Therefore, $\sigma_{n,m}^k$ is well defined and smooth in $x$.

    The proof is divided into three steps. 
    In Step~1, we estimate $C_{n,m}(t)$. 
    In Step~2, we show that for each $(n,m,k)\in \z_+^3$, $\sigma_{n,m}^k$ satisfies~\ref{item:A2} and Assumption~\ref{assu:H}.
    In Step~3, we verify (i) and (ii) by selecting suitable indices $(m,k) = (m(n), k(n))$. 

    \smallskip
    \textbf{Step 1:} For any $(x,t)\in D_{n+1}\times [r,T]$, we rewrite $f_m$ as 
    $$f_m=\int_\r (\tilde\sigma_{m}+\sigma)(\tilde\sigma_{m}-\sigma)\dd s-\int_{|s|> m} \tilde\sigma_{m}^2\dd s.$$
    Then, the Cauchy--Schwarz inequality yields
    \begin{align*}
        f_{m}^2\le 2\int_\r (\tilde\sigma_m+\sigma)^2\dd s \cdot\int_\r (\tilde\sigma_m-\sigma)^2\dd s+2\Bigl(\int_{|s|>m} \tilde\sigma_m^2 \dd s \Bigr)^2=:2\Sigma_1\cdot\Sigma_2+2\Sigma_3^2.
    \end{align*}
    By condition~\ref{item:1c} and~\eqref{eq:a4}, there exists $C>0$, independent of $n,m$, such that
    \begin{align*}
    \Sigma_1 
    \le 
    2 F_0(x,t)
    + 2\bar \sigma(x,t)
    \le C (1+|x|)^2.
    \end{align*}
    Setting $N=n+2$, by~\eqref{Sig_22}, $\Sigma_2$ is bounded by $\int_{\r} \eta(v) \rho_{N}(|v|/m)^2\dd v$.
    Finally,
    \begin{align}\label{sig_3}
        \nonumber \Sigma_3\le &\int_{|s|> m}\sigma(x,t,s)^2\dd s+\int_\r |\tilde\sigma_m(x,t,s)^2-\sigma(x,t,s)^2| \dd s\\
        \le &\int_{|s|> m}\sigma(x,t,s)^2\dd s+\sqrt{\Sigma_1 \Sigma_2}.
    \end{align}

    \noindent
    Let $h_m(x,t):=\int_{|s|> m}\sigma(x,t,s)^2\dd s$. Since $\int_\r\sigma(x,t,s)^2\dd s <\infty$, we have $h_m(x,t)\downarrow 0$ as $m\to\infty$ for each $(x,t)$. By condition~\ref{item:1b},
    \begin{align*}
        \left|h_m(x,t)-h_m(y,t)\right|^2
        \le &\int_\r (\sigma(x,t,s)+\sigma(y,t,s))^2 \dd s \cdot \int_\r (\sigma(x,t,s)-\sigma(y,t,s))^2 \dd s\\
        \le & C\bigl[(1+|x|)^2+(1+|y|)^2\bigr] \rho_N(|x-y|)^2.
    \end{align*}
    Combining~\eqref{assu:conti_t} yields that $\{h_m\}_{m\ge1}$ is equicontinuous and uniformly bounded.
    By the Arzel\`a--Ascoli theorem, $h_m(x,t) \downarrow 0$ uniformly on $D_{n+1}\times [r,T]$. It follows that $C_{n,m}^2(t)$ admits the upper bound (up to a constant)
    \begin{align}\label{C_m_m}
       \frac{1}{m^2}+(n+2)^2 \int_{\r} \eta(v) \rho_{N}(|v|/m)^2\dd v + \sup_{x\in D_{n+1}} h_m(x,t)^2.
    \end{align}
    Moreover, the right-hand side converges to $0$ uniformly in $t$ as $m\to\infty$.

    \smallskip
    \textbf{Step 2:} Fix $n,m,k\in \z_+$. By Step 1, there exists a constant $C$ such that $C_{n,m}(t)\le C$ for all $t\in [r,T]$. Set $\hat{\sigma}(x,t,s) = \tilde{\sigma}_{n,m}^k(x,t,s)$, then
    \begin{align}\label{eq:F_split}
        \int_{\r} (\partial_x^\gamma \hat{\sigma}(x,t,s))^2 \dd s
        = \underbrace{\int_{|s|\le m} (\partial_x^\gamma \tilde{\sigma}_m(x,t,s))^2 \dd s}_{\le F_\gamma^m(x,t)} + \underbrace{ \left(\partial_x^\gamma \sqrt{\Phi_{n,m}^k(x,t)}\right)^2 }_{=:H_\gamma(x,t)}.
    \end{align}

    \noindent
    When $\gamma=0$, combining~\eqref{eq:a4} and~\eqref{eq:F_split}, we obtain that 
    \begin{align*}
        \left|\int_\r \sigma_{n,m}^k(x,t,s) \sigma_{n,m}^k(y,t,s) \dd s \right|\le &\; \left|\int_\r \hsig(x,t,s) \hsig(y,t,s) \dd s \right|\\
        \le &\; \sqrt{F_0^m(x,t)+H_0(x,t)}\, \sqrt{F_0^m(y,t)+H_0(y,t)}\\
        \le &\; \sqrt{2C_0 K_\sigma (2+|x|)^2+2C}\, \sqrt{2C_0 K_\sigma (2+|y|)^2+2C}\\
        \le &\; (8C_0K_\sigma+2C) (1+|x|)(1+|y|).
    \end{align*}
    This yields the first bound in~\ref{item:A2}. Moreover, for $\alpha=1,2,3$, the Leibniz rule gives
    \begin{equation*}
        (\partial_x^\alpha\sigma_{n,m}^k(x,t,s))^2 \le \binom{2\alpha}{\alpha}\sum_{\beta=0}^\alpha \bigl|\zeta_{n/2}^{(\beta)}(x)\bigr|^2 (\partial_x^{\alpha-\beta} \hat{\sigma}(x,t,s))^2.
    \end{equation*}
    Since $\zeta_{n/2}$ and its derivatives are bounded, and~\eqref{eq:a4} ensures that $F_\gamma(x,t)$ is bounded, uniform control follows from the boundedness of $H_\gamma(x,t)$ for $0 \le \gamma \le \alpha$. Each derivative $\partial_x^\gamma \sqrt{\Phi_{n,m}^k(x,t)}$ is a combination of terms $\Phi_{n,m}^k$ and its derivatives in the numerator, divided by $(\Phi_{n,m}^k)^{p}$ (where $p \ge 1/2$) in the denominator. Since $\Phi_{n,m}^k$ is bounded away from zero by $1/m$, these derivatives remain bounded on $D_n$.

    Finally, we verify Assumption~\ref{assu:H} by establishing~\eqref{assu:conti_t} for $\sigma_{n,m}^k$. By~\eqref{eq:conti_sig_n_t} and the uniform continuity of the square-root function, it suffices to show that for any compact sets $K_1, K_2$,
    \begin{align*}
        \sup_{x \in K_1}\sup_{t \in K_2} \left|\Phi_{n,m}^k(x,t+h) - \Phi_{n,m}^k(x,t)\right| \to 0,\qquad \text{as } h \to 0.
    \end{align*}
    This follows from~\eqref{assu:conti_t} and~\eqref{eq:conti_sig_n_t}, which imply that $f_m(x,t+h)-f_m(x,t)\to0$ uniformly on $K_1\times K_2$ as $h\to0$, and hence also the continuity of $C_{n,m}(t)$.

    \smallskip
    \textbf{Step 3:}
    We first verify (i). 
    Fix a compact set $K\subset \r$ and choose $N\ge 1$ such that $K\subset D_N$. Then, for any $x\in K$, $t\in [r,T]$ and $n\ge 2N$, 
    \begin{multline*}
    \qquad \int_\r \bigl(\sigma_{n,m}^k(x,t,s)-\sigma(x,t,s)\bigr)^2 \dd s = \int_\r \bigl(\tilde\sigma_{n,m}^k(x,t,s)-\sigma(x,t,s)\bigr)^2 \dd s\\
    \le 2\int_\r \bigl(\tilde\sigma_{n,m}^k(x,t,s)-\tilde\sigma_m(x,t,s)\bigr)^2 \dd s+ 2\Sigma_2 \le 4\Sigma_3+8C_{n,m}(t)+2\Sigma_2.\qquad 
    \end{multline*}
    By~\eqref{Sig_22},~\eqref{sig_3}, and~\eqref{C_m_m}, we have
    \begin{align}\label{lim_m}
        \lim_{m\to\infty}\sup_{k}\sup_{r\le t\le T}\sup_{x\in K}\int_{\r}\bigl(\sigma_{n,m}^k(x,t,s)-\sigma(x,t,s)\bigr)^2 \dd s=0.
    \end{align}
    
    Moreover, define $\bar\sigma_{n,m}^k(x,t):=\int_\r \sigma_{n,m}^k(x,t,s)^2\dd s$. Then, by~\eqref{def:sig_k},
    \begin{align}\label{eq:partial_sig_k}
        \bar\sigma_{n,m}^k(x,t)=\zeta_{n/2}(x)^2 \left[\int_{|s|\le m} \tilde\sigma_m(x,t,s)^2 \dd s +\eta_k* g_{n,m}(x,t) \right].
    \end{align}
    For any $n\ge 2N$, the standard mollifier approximation yields
    \begin{align}\label{lim_k}
       \nonumber &\lim_{k\to \infty}\sup_{r\le t\le T}\sup_{x\in K} |\partial_x \bigl[\bar\sigma_{n,m}^k(x,t)-\bar\sigma(x,t)\bigr]|\\
       =&\; \lim_{k\to \infty}\sup_{r\le t\le T}\sup_{x\in K} |\partial_x\! \left[-g_{n,m}(x,t)+\eta_k*g_{n,m}(x,t)\right]| = 0.
    \end{align}
    
    By~\eqref{eq:partial_sig_k}, for all $n,m,k$, $\left|\partial_x\bar\sigma_{n,m}^k(x,t) \right|$ is bounded by
    \begin{align*}
        &\quad \frac{4B}{n}\,\mathbbm{1}_{\{|x|\le n\}}
        \Bigl[\int_{|s|\le m} \tilde \sigma_m^2 \dd s+\eta_k*g_{n,m}\Bigr]+\mathbbm{1}_{\{|x|\le n\}}\bigl|\partial_x \bar\sigma - \partial_x g_{n,m}+ \partial_x (\eta_k * g_{n,m})\bigr| \\
        &\le \mathbbm{1}_{\{|x|\le n\}}\left\{4B\Bigl[\tfrac{C}{2n}(1+|x|)^2 + \tfrac{2C_{n,m}(t)}{n}\Bigr]+K_{\sigma}(1+|x|)+\bigl|- \partial_x g_{n,m}+ \partial_x (\eta_k * g_{n,m})\bigr|\right\} \\
        &\le \mathbbm{1}_{\{|x|\le n\}}\left\{(4BC+K_\sigma)(1+|x|) + \tfrac{8BC_{n,m}(t)}{n}+\bigl|- \partial_x g_{n,m}+ \partial_x (\eta_k * g_{n,m})\bigr|\right\}.
    \end{align*}
    Here, the first inequality follows from~\eqref{eq:a4} and condition~\ref{item:1c}. For each $n$, by Step~1 and~\eqref{lim_m}, choose $m=m(n)$ such that $\sup_{t\in [r,T]}\tfrac{8BC_{n,m}(t)}{n}\le 2^{-n}$ and the quantity appearing in~\eqref{lim_m} is bounded by $2^{-n}$. For this choice of $m(n)$, it follows from~\eqref{lim_k} that one can further choose $k=k(n)$ such that $|- \partial_x g_{n,m}+ \partial_x (\eta_k * g_{n,m})|\le 2^{-n}$ for all $(x,t)\in D_n\times [r,T]$. Consequently,~\eqref{eq:assume_sig_b} and~\eqref{eq:assume_sig_partial} hold for $\sigma_n:=\sigma_{n,m(n)}^{k(n)}$ and
    \begin{align*}
        \left|\partial_x\bar\sigma_{n,m(n)}^{k(n)} \right|\le (4BC+K_\sigma+1)(1+|x|).
    \end{align*}
    Hence, (i) and (ii) follow. This completes the proof.
\end{proof}

\begin{lemma}\label{lm:C2}
    Suppose that $(\sigma,b)$ satisfy Assumption~\ref{assu:H} and conditions~\ref{item:1a}--\ref{item:1d} on $I$. Let $N\ge 2/|I|+1$. Then there exists a sequence $(\sigma_n,b_n)_{n\ge N}$ satisfying Assumption~\ref{assu:H} and conditions~\ref{item:1a}--\ref{item:1d} on $\r$, such that,
    \begin{align}\label{eq:b_n_sig_n}
        b_n(x,t)=b(x,t)\ \text{ and }\ \sigma_n(x,t,\cdot)=\sigma(x,t,\cdot), \quad \forall x\in I_n,\, t\in \r,
    \end{align}
    where $I_n:=\bigl(\inf I+\tfrac{1}{n},\, \sup I-\tfrac{1}{n}\bigr)$, with the convention $-\infty+c=-\infty$ and $+\infty-c=+\infty$. 
\end{lemma}

\begin{proof}
    Let $\psi$ be a smooth step function with $\psi(u)=0$ for $u\le 0$ and $\psi(u)=1$ for $u\ge 1$. Set $a=\inf I$ and $b=\sup I$, and fix $x_0\in I_n$ for $n\ge N$. Define
    $$
    \chi_n(x) := x_0 + \int_{x_0}^x \psi(n(u-a)) \psi(n(b-u))  \dd u,
    $$
    with the convention $\psi(+\infty) = 1$. Then $\chi_n$ is smooth, nondecreasing, and maps $\r$ into $I$, with $0 \le \chi_n' \le 1$. Moreover, $\chi_n(x)=x$ for all $x\in I_n$.
    Since $\chi_n''$ is supported on $\bar I\setminus I_n$, $\chi_n'$ is Lipschitz; that is, there exists $L_n>0$ such that $|\chi_n'(x)-\chi_n'(y)| \le L_n |x-y|$ for all $x,y \in \r$. The bound $|\chi_n'|\le 1$ further implies $|\chi_n(x)|\le C(1+|x|)$, where $C := \max\{1, 2|x_0|\}$ is independent of $n$. For any $x,t,s\in \r$, set
    $$
    b_n(x,t):=\chi_n'(x)b\big(\chi_n(x),t\big),\qquad \sigma_n(x,t,s):=\sigma\big(\chi_n(x),t,s\big).
    $$

    By construction, both~\eqref{eq:b_n_sig_n} and Assumption~\ref{assu:H} hold. It remains to verify that the pair $(\sigma_n,b_n)$ satisfies \ref{item:1a}--\ref{item:1d} on $\r$. Fix a finite time interval $[T_1,T_2]$ and let $t\in[T_1,T_2]$. First, condition~\ref{item:1a} follows from
    $$
    |b_n(x,t)|\le K_b(1+|\chi_n(x)|)\le K_b(1+C(1+|x|))\le 2CK_b(1+|x|),\qquad \forall x\in \r.
    $$
    To check~\ref{item:1b}, fix $m\ge 1$ and choose $M=M(m,n)$ such that $\chi_n([-m,m])\subset [-M,M]\cap I$. By~\ref{item:1b} for $(\sigma,b)$, there is a nonnegative nondecreasing function $\rho_M$ on $\r_+$ such that 
    $$
    \int_\r|\sigma_n(x,t,s)-\sigma_n(y,t,s)|^2\dd s\le \rho_M(|\chi_n(x)-\chi_n(y)|)^2\le \rho_M(|x-y|)^2,
    $$
    for all $x, y \in [-m,m]$, and $\int_{0_+} \rho_M(z)^{-2} \dd z=\infty$. As for~\ref{item:1c}, noting that $\bar\sigma(\chi_n(x),t)=\bar\sigma_n(x,t)$, the chain rule combined with condition~\ref{item:1c} for $(\sigma,b)$ implies that
    \begin{align*}
        |\bar\sigma_n'(x,t)|\le |\partial_x \bar\sigma(\chi_n(x),t)||\chi_n'(x)|\le K_\sigma (1+|\chi_n(x)|)\le 2CK_\sigma (1+|x|),\qquad \forall x\in \r.
    \end{align*}
    Finally, set $b^*:=-b+\tfrac12\partial_x\bar\sigma$. Then
    $$
    b_n^*:=-b_n+\tfrac12\partial_x\bar\sigma_n
    =\chi_n' b^*(\chi_n(\cdot),\cdot).
    $$
    For each $m\ge1$, choose $M=M(m,n)$ as above. By condition~\ref{item:1d} for $(\sigma,b)$, for all $x,y\in[-m,m]$ and $t\in[T_1,T_2]$,
    \begin{align*}
        |b_n(x,t)-b_n(y,t)|
        &\le r_M(|x-y|)+C_{m,n}|x-y|,\\
        |b_n^*(x,t)-b_n^*(y,t)|
        &\le r'_M(|x-y|)+C_{m,n}|x-y|,
    \end{align*}
    where $C_{m,n}:=(K_b+\tfrac12K_\sigma)(1+M)L_n$. Let
    $$
    \tilde r_{m,n}(z):=r_M(z)+C_{m,n}z,\qquad
    \tilde r'_{m,n}(z):=r'_M(z)+C_{m,n}z.
    $$
    Both functions are nondecreasing and concave. Furthermore, 
    $$
    \int_{0_+}\tilde r_{m,n}(z)^{-1}\dd z=\int_{0_+}\tilde r'_{m,n}(z)^{-1}\dd z=\infty.
    $$
    Indeed, since $r_M$ is concave and $r_M(0)=0$, the map $z \mapsto r_M(z)/z$ is nonincreasing on $(0,\infty)$. This guarantees that on a sufficiently small interval $(0, \delta)$, either $r_M(z)\ge C_{m,n}z$ throughout or $r_M(z)\le C_{m,n}z$ throughout. Hence, $\tilde r_{m,n}(z)$ is uniformly bounded from above by either $2r_M(z)$ or $2C_{m,n}z$ on $(0,\delta)$. Since $\int_{0_+}r_M(z)^{-1}\dd z=\int_{0_+}z^{-1}\dd z=\infty$, we obtain $\int_{0_+}\tilde r_{m,n}(z)^{-1}\dd z=\infty$. The same argument applies to $\tilde r'_{m,n}$. Thus condition~\ref{item:1d} holds.
\end{proof}

\begin{remark}\label{rmk:C2}
    As in the proof of the preceding lemma, if $(\sigma,b)$ satisfy Assumption~\ref{assu:H} and conditions~\ref{item:1a}, \ref{item:1b}, \ref{item:2c}, and \ref{item:2d} on $I$, then there exists a sequence $(\sigma_n,b_n)_{n\ge N}$ satisfying Assumption~\ref{assu:H} and the same conditions on $\r$, such that~\eqref{eq:b_n_sig_n} holds on $I_n$.
\end{remark}

\section{Pathwise uniqueness}\label{app:path_unique}
This section is mainly devoted to pathwise uniqueness for~\eqref{eq:SDE_gener} and~\eqref{eq:back_flow}. The proof follows the standard argument for one-dimensional SDEs with non-Lipschitz coefficients (see, e.g., \cite[Chapter~IV, Theorem~3.2]{SDE81}), and is included for completeness. We also record the proof of Lemma~\ref{app:lem_conti}.

\begin{lemma}\label{app:lem_conti}
    Suppose that Assumption~\ref{assu:H} and conditions~\ref{item:1a}, \ref{item:1b}, \ref{item:2c}, and~\ref{item:2d} hold. Then the map $(x,t)\mapsto \bar\sigma(x,t)$ is continuous. Moreover, if conditions~\ref{item:1c} and~\ref{item:1d} hold, then $(x,t)\mapsto -b(x,t)+\tfrac12\partial_x\bar\sigma(x,t)$ is also continuous on $\r^2$.
\end{lemma}

\begin{proof}
    Let $(x_n,t_n)\to(x,t)$. Then there exist an integer $m\ge 1$ and compact sets $K_1,K_2$ such that $x_n,x\in K_1\subset[-m,m]$ and $t_n,t\in K_2$. By the triangle inequality,
    \begin{align}\label{eq:trian}
    |\bar\sigma(x_n,t_n)-\bar\sigma(x,t)|
    \le |\bar\sigma(x_n,t_n)-\bar\sigma(x,t_n)|
    +|\bar\sigma(x,t_n)-\bar\sigma(x,t)|.
    \end{align}
    By conditions~\ref{item:1b} and~\ref{item:2c}, and the Cauchy--Schwarz inequality,
    \begin{align*}
    |\bar\sigma(x_n,t_n)-\bar\sigma(x,t_n)|^2
    \le 2\bigl(\bar\sigma(x_n,t_n)+\bar\sigma(x,t_n)\bigr)\rho_m(|x_n-x|)^2 \to 0,
    \end{align*}
    as $x_n\to x$. Together with~\eqref{assu:conti_t} and~\eqref{eq:trian}, this implies that $(x,t)\mapsto \bar\sigma(x,t)$ is continuous.

    Since Assumption~\ref{assu:H} requires continuity of $b(x,t)$, it remains to prove continuity of $\partial_x\bar\sigma$. Applying the triangle inequality again,
    \begin{align*}
    |\partial_x\bar\sigma(x_n,t_n)-\partial_x\bar\sigma(x,t)|
    &\le |\partial_x\bar\sigma(x_n,t_n)-\partial_x\bar\sigma(x,t_n)|
    +|\partial_x\bar\sigma(x,t_n)-\partial_x\bar\sigma(x,t)|.
    \end{align*}
    The first term vanishes by condition~\ref{item:1d}. To treat the second term, we use the Arzelà--Ascoli theorem. By conditions~\ref{item:1c} and~\ref{item:1d}, the family $\{\partial_x\bar\sigma(\cdot,t)\}_{t\in K_2}$ is uniformly bounded and equicontinuous on $K_1$. Hence, $(t_n)$ admits a further subsequence $(t_{n_k})$ such that $\partial_x\bar\sigma(\cdot,t_{n_k})$ converges uniformly on $K_1$ to a continuous function $g$.
    For any $y\in K_1$,
    \begin{align*}
    \bar\sigma(y,t_{n_k})-\bar\sigma(0,t_{n_k})
    =\int_0^y \partial_z \bar\sigma(z,t_{n_k})\,\dd z.
    \end{align*}
    Letting $k\to\infty$ yields $\bar\sigma(y,t)-\bar\sigma(0,t)=\int_0^y g(z)\dd z$, and differentiating with respect to $y$ reveals $g(y)=\partial_x\bar\sigma(y,t)$. Since every subsequence has the same limit, the whole sequence converges, hence $\partial_x\bar\sigma(x,t_n)\to\partial_x\bar\sigma(x,t)$. This completes the proof.
\end{proof}

In the following, we assume that $(\sigma,b)$ satisfy Assumption~\ref{assu:H} and conditions~\ref{item:1a}, \ref{item:1b}, \ref{item:2c}, and~\ref{item:2d}. Fix $(a,r)\in\r^2$, and let $X$ and $Y$ be two solutions to~\eqref{eq:SDE_gener} driven by the same white noise $\cW$. Our goal is to establish pathwise uniqueness, namely,
\begin{align*}
\mathbb{P}\bigl(X_x=Y_x,\ \forall x\ge r\bigr)=1.
\end{align*}
Fix $T>r$ and define the stopping times
$$
\tau_m := \inf\bigl\{ x \in [r,T] :\, |X_x| \ge m \text{ or } |Y_x| \ge m \bigr\} \wedge T.
$$
Since $\tau_m\uparrow T$ almost surely by Lemma~\ref{lem:range}, it remains to prove that for every integer $m\ge |a|+1$, one has $X_x=Y_x$ for all $x\in[r,\tau_m]$ almost surely. 

\smallskip
Set $Z_x:=X_x-Y_x$ for $x\in[r,\tau_m]$. Then $Z_r=0$, and for every $x\in[r,\tau_m]$, $Z$ satisfies
\begin{equation*}
	Z_x = \int_r^x\!\! \int_{\r} \eta(u,s) \cW(\dd s, \dd u) + \int_r^x \beta(u) \dd u,
\end{equation*}
where $\eta(u,s) := \sigma(X_u, u, s) - \sigma(Y_u, u, s)$ and $\beta(u) := b(X_u, u) - b(Y_u, u)$. Since both $X$ and $Y$ remain in $[-m,m]$ on $[r,\tau_m]$, by conditions~\ref{item:1b} and~\ref{item:2d}, there exist nondecreasing functions $\rho,\kappa$ on $\r_+$, with $\rho$ nonnegative and $\kappa$ concave, satisfying $\int_{0_+} \rho(z)^{-2} \dd z=\int_{0_+} \kappa(z)^{-1} \dd z=\infty$ such that for any $u\in[r,\tau_m]$,
\begin{align}\label{eq:boun_app_c}
    \int_{\r} \eta(u,s)^2 \dd s \leq \rho(|Z_u|)^2, \qquad |\beta(u)| \leq \kappa(|Z_u|).
\end{align}

\noindent
Let $\{a_n\}_{n \ge 0}\subset \r_+$ be a decreasing sequence such that $a_0 = 1$, $a_n \downarrow 0$, and for each $n \ge 1$, $\int_{a_n}^{a_{n-1}}\! \rho(z)^{-2} \dd z = n$.
For each $n\ge1$, choose a continuous function $\psi_n$ on $\r$ supported in $(a_n,a_{n-1})$ such that $\int_{a_n}^{a_{n-1}}\!\psi_n(z)\dd z=1$, and $0\le\psi_n(z)\le \frac{2}{n\rho(z)^2}$ for all $z\in(a_n,a_{n-1})$.
Define $\phi_n\in C^2(\r)$ by
$$
\phi_n(x):=\int_0^{|x|}\!\! \int_0^y \psi_n(z)\dd z \dd y, \qquad x\in\r.
$$
For $x\in[r,\tau_m]$, applying It\^o's formula to $\phi_n(Z_x)$ gives
\begin{align*}
	\phi_n(Z_x) &= \int_r^x \phi_n'(Z_u) \beta(u) \dd u + \int_r^x\!\! \int_{\r} \phi_n'(Z_u) \eta(u,s) \cW(\dd s, \dd u) \\
	&\quad + \frac{1}{2} \int_r^x \phi_n''(Z_u) \int_{\r} \eta(u,s)^2 \dd s \dd u. 
\end{align*}
Taking expectations on both sides, we obtain
\begin{align}\label{eq:ito}
	\e[\phi_n(Z_x)] = \e \!\left[\int_r^x \phi_n'(Z_u) \beta(u) \dd u\right] + \frac{1}{2} \e\! \left[ \int_r^x \phi_n''(Z_u) \int_{\r} \eta(u,s)^2 \dd s \dd u\right] =: I_1+I_2.
\end{align}
By~\eqref{eq:boun_app_c} and $|\phi_n'(\cdot)|\le 1$, we have $|I_1|\le \int_r^x\e \!\left[\kappa(|Z_u|) \right] \dd u\le \int_r^x \kappa(\e [|Z_u|]) \dd u$, where the second inequality follows from Jensen's inequality and the concavity of $\kappa$. Moreover, by~\eqref{eq:boun_app_c} and the estimate $\phi_n''(x)=\psi_n(|x|)\le \frac{2}{n\rho(|x|)^2}\mathbbm{1}_{\{a_n<|x|<a_{n-1}\}}$, it holds that $|I_2|\le \frac{x-r}{n} \le \frac{T-r}{n}$. Letting $n\to \infty$ in~\eqref{eq:ito}, and noting that $\phi_n(x)\uparrow |x|$ as $n \to \infty$, we obtain by the monotone convergence theorem that $\e [|Z_x|]\le \int_r^x \kappa(\e [|Z_u|]) \dd u$. Thus, by Osgood's uniqueness theorem (see, e.g., \cite[pp.~12--13]{ODE93}), $\e[|Z_x|] = 0$ for all $x \in [r, \tau_m]$, which implies $Z_x = 0$ almost surely for all $x \in [r, \tau_m]$. Therefore, pathwise uniqueness holds, and the Yamada–Watanabe theorem yields the unique strong solution.

\bibliographystyle{abbrvurl}
\bibliography{duality}

\end{document}